\documentclass{amsproc}
\usepackage[margin=1.2in,nomarginpar]{geometry}
\usepackage{amssymb}
\usepackage{amsmath}
\usepackage{mathdots}
\usepackage{amsbsy}
\usepackage{amscd}
\usepackage{amsthm}

\usepackage{url}
\usepackage{xcolor}
\usepackage{amsmath,amssymb,amsthm,amsfonts,longtable}
\usepackage[english]{babel}
\usepackage{tikz-cd}
\usetikzlibrary{cd}
\usetikzlibrary{decorations.markings}
\tikzset{negated/.style={
		decoration={markings,
			mark= at position 0.5 with {
				\node[transform shape] (tempnode) {$\times$};
			}
		},
		postaction={decorate}
	}
}
\usepackage{array}
\usepackage[colorlinks,citecolor=red,urlcolor=blue,bookmarks=false,hypertexnames=true]{hyperref} 
\usepackage{multirow}
\usepackage{pbox}
\newtheorem{theorem}{Theorem}
\newtheorem{corollary}[theorem]{Corollary}
\newtheorem{lemma}[theorem]{Lemma}

\newtheorem{remark}[theorem]{Remark}

\newtheorem{example}[theorem]{Example}
\newtheorem{definition}[theorem]{Definition}
\newcommand{\Irr}{\textnormal{Irr}}

\newcommand{\cl}{\textnormal{cl}}
\newcommand{\nil}{\textnormal{c}}
\newcommand{\cd}{\textnormal{cd}}
\newcommand{\nl}{\textnormal{nl}}
\newcommand{\lin}{\textnormal{lin}}

\newcommand{\gal}{\textnormal{Gal}}
\newcommand{\aut}{\textnormal{Aut}}

\newcommand{\Syl}{\textnormal{Syl}}
\newcommand{\order}{\textnormal{Order}}

\title[]{A Combinatorial Technique for the Wedderburn Decomposition of Rational Group Algebras of Nested GVZ $p$-groups}
\author{Ram Karan Choudhary$^*$}
\address{Indian Institute of Technology, Bhubaneswar, Arugul Campus, Jatni, Khurda-752050, India.}
\email{ramkchoudhary1997@gmail.com}
\author{Sunil Kumar Prajapati}
\address{Indian Institute of Technology, Bhubaneswar, Arugul Campus, Jatni, Khurda-752050, India.}
\email{skprajapati@iitbbs.ac.in}
\thanks{$^{\textbf{*}}$ Corresponding author.
}
\begin{document}
	
\begin{abstract}
	In this article, we present a combinatorial formula for the Wedderburn decomposition of rational group algebras of nested GVZ $p$-groups, where $p$ is an odd prime. Using this formula, we derive an explicit combinatorial expression for the Wedderburn decomposition of rational group algebras of all two-generator $p$-groups of class $2$. Additionally, we provide explicit combinatorial formulas for the Wedderburn decomposition of rational group algebras of certain families of nested GVZ $p$-groups with arbitrarily large nilpotency class. We also classify all nested GVZ $p$-groups of order at most $p^5$ and compute the Wedderburn decomposition of their rational group algebras. Finally, we determine a complete set of primitive central idempotents for the rational group algebras of nested GVZ $p$-groups.
\end{abstract}
	\maketitle

	\section{Introduction} 
	Throughout this paper, all groups are finite. For a group $G$, we write $\Irr(G)$ for the set of irreducible complex characters of $G$. Let $N \trianglelefteq G$ and $\chi \in \Irr(G)$. We say that $\chi$ is \emph{fully ramified} over $N$ if $\chi(g)=0$ for all $g \in G \setminus N$. A group $G$ is called of \emph{central type} if there exists $\chi \in \Irr(G)$ such that $\chi$ is fully ramified over $Z(G)$. This class of groups was first introduced by DeMeyer and Janusz~\cite{DJ}, who showed that a group $G$ is of central type if and only if each Sylow $p$-subgroup $\Syl_p(G)$ of $G$ is of central type and $Z(\Syl_p(G)) = Z(G) \cap \Syl_p(G)$. It is also known that groups of central type are solvable (see~\cite{HI}). Central type groups have been studied extensively in the literature, including~\cite{DJ, Espu, Gagola, HI}. Motivated by this definition, one says that a character $\chi \in \Irr(G)$ is of \emph{central type} if $\chi$ vanishes on $G \setminus Z(\chi)$, where 
	\[
	Z(\chi) = \{ g \in G : |\chi(g)| = \chi(1) \}
	\]
	is the center (or quasi-kernel) of $\chi$. Equivalently, $\chi$ is of central type if and only if $\bar{\chi} \in \Irr(G/\ker(\chi))$ is fully ramified over $Z(G/\ker(\chi))$, where $\chi$ is the lift of $\bar{\chi}$ to $G$, i.e., $G/\ker(\chi)$ is a group of central type with faithful character $\bar{\chi}$ (see \cite{LewisGVZ2}). Groups in which every irreducible complex character is of central type are called \emph{GVZ-groups}. GVZ-groups were first investigated in~\cite{Onogroup} under the name \emph{groups of Ono type}. For a group $G$, a conjugacy class $\mathcal{C}$ is said to be of \emph{Ono type} if for every $\chi \in \Irr(G)$ and every $g \in \mathcal{C}$, either $\chi(g)=0$ or $|\chi(g)|=\chi(1)$. A group $G$ is said to be of Ono type if every conjugacy class of $G$ is of Ono type. Groups of Ono type were first introduced by Ono~\cite{Ono}. A group $G$ is called \emph{nested} if for all characters $\chi, \psi \in \Irr(G)$, we have either $Z(\chi) \subseteq Z(\psi)$ or $Z(\psi) \subseteq Z(\chi)$. A GVZ-group $G$ is called a \emph{nested GVZ-group} if $G$ is nested. Note that a nested GVZ-group $G$ is a GVZ-group satisfying $Z(\psi) \subseteq Z(\chi)$ whenever $\chi(1) \leq \psi(1)$ for $\chi, \psi \in \Irr(G)$ (see \cite[Lemma 7.1]{LewisGVZ2}). Moreover, such a group $G$ is strictly nested by degrees, i.e.,
	\[
	Z(\psi) \subset Z(\chi) \quad \Longleftrightarrow \quad \chi(1) < \psi(1), \quad \text{for } \chi, \psi \in \Irr(G).
	\]
    A group $G$ has an irreducible character $\chi$ with $\chi(1) = |G/Z(\chi)|^{1/2}$
	if and only if $\chi(g)=0$ for all $g \in G \setminus Z(\chi)$ (see \cite[Corollary 2.30]{I}). Thus, the latter definition of a nested GVZ-group is equivalent to the one used by Nenciu in~\cite{NenciuGVZ, NenciuGVZ2}, which was motivated by problems posed by Berkovich~\cite{Berkovich} (Research Problems 24 and 30 therein). Moreover, a nested GVZ-group is nilpotent and isomorphic to the direct product of a nested GVZ $p$-group and an abelian group, where $p$ is a prime (see \cite[Corollary~2.5]{NenciuGVZ}). For further background on such groups, see~\cite{Burkett2, LewisGVZ, LewisGVZ3, Burkett, LewisGVZ2, NenciuGVZ, NenciuGVZ2}. In this article, we investigate the Wedderburn decomposition of rational group algebras of nested GVZ $p$-groups, where $p$ is an odd prime.

	Let $G$ be a finite group, and let $\mathbb{F}$ a field. According to the Wedderburn-Artin theorem, the group algebra $\mathbb{F}G$ is semisimple if and only if it decomposes into a direct sum of matrix algebras over division rings. More precisely, there is an isomorphism
	\[
	\mathbb{F}G \cong \bigoplus_{i=1}^{r} M_{n_i}(D_i),
	\]
	where each $M_{n_i}(D_i)$ is a full matrix ring of size $n_i$ over a division ring $D_i$ that is finite-dimensional over its center. These rings $M_{n_i}(D_i)$ are called the \emph{simple components} of $\mathbb{F}G$. Furthermore, by the Brauer-Witt theorem (see \cite{Yam}), each simple component is Brauer equivalent to a cyclotomic algebra. The Wedderburn decomposition of rational group algebras has garnered significant attention in recent years due to its importance in understanding various algebraic structures (see \cite{Herman, Jes-Rio, Rit-Seh}). Recent studies have employed Shoda pair theory to analyze this decomposition, as documented in several works including \cite{BM14, BGO, Jes-Lea-Paq, Jes-Olt-Rio, Olt07}. Mathematical software implementations, most notably the \textsc{Wedderga} package in {\sf GAP}~\cite{Gap}, have been developed based on these methods. However, exact computations often remain challenging, especially for groups of large order. For finite abelian groups, Perlis and Walker \cite{PW} provided a well-known combinatorial formula for the Wedderburn decomposition of their rational group algebras.
	
	\begin{theorem}\cite[Theorem 1]{PW}\label{Perlis-walker}
		Let $G$ be a finite abelian group of exponent $m$. Then the Wedderburn decomposition of $\mathbb{Q}G$ is given by
		\[\mathbb{Q}G \cong \bigoplus_{d|m} a_d \mathbb{Q}(\zeta_d),\]
		where $a_d$ is equal to the number of cyclic subgroups of $G$ of order $d$.
	\end{theorem}
	 A non-abelian group $G$ is called a \emph{VZ-group} if every $\chi \in \nl(G)$ is fully ramified over $Z(G)$. It is easy to see that a VZ-group is a special case of a nested GVZ-group. In \cite{Ram}, we formulate the computation of the Wedderburn decomposition of rational group algebra of a VZ $p$-group $G$ (where $p$ is an odd prime), solely based on computing the number of cyclic subgroups of $G/G'$, $Z(G)$ and $Z(G)/G'$, which is similar to Theorem~\ref{Perlis-walker}.
	 \begin{theorem}\cite[Theorem 1]{Ram}\label{lemma:Wedderburn VZ}
	 	Let $G$ be a finite VZ $p$-group, where $p$ is an odd prime. Let $m$ and $m'$ denote the exponents of $Z(G)$ and $Z(G)/G'$, respectively. Then the Wedderburn decomposition of $\mathbb{Q}G$ is given by
	 	$$\mathbb{Q}G \cong \mathbb{Q}(G/G')\;\bigoplus_{\substack{d\mid m\\ d \nmid m'}} a_dM_{|G/Z(G)|^{1/2}}\!\left(\mathbb{Q}(\zeta_d)\right)\;\bigoplus_{\substack{d|m\\ d|m'}}(a_d-a_d')M_{|G/Z(G)|^{1/2}}\!\left(\mathbb{Q}(\zeta_d)\right),$$
	 	where $a_d$ and $a_d'$ are the number of cyclic subgroups of order $d$ of $Z(G)$ and $Z(G)/G'$, respectively.
	 \end{theorem}
	 
	 This work extends Theorem~\ref{lemma:Wedderburn VZ} by deriving a combinatorial formula for the Wedderburn decomposition of rational group algebras of nested GVZ $p$-groups, where $p$ is an odd prime. 
	 
	 \begin{theorem}\label{thm:WedderburnGVZ}
	 	Let $G$ be a finite nested GVZ $p$-group, where $p$ is an odd prime.  
	 	Let $\cd(G) = \{p^{\delta_i} : 0 \leq i \leq n,\; 0=\delta_0 < \delta_1 < \delta_2 < \cdots < \delta_n\}$, and for $0 \leq i \leq n$ define $Z_{\delta_i} := Z(\chi)$ for some $\chi \in \Irr_{p^{\delta_i}}(G)$. For $1 \leq r \leq n$, let $m_r$ and $m_r'$ denote the exponents of $Z_{\delta_r}/[Z_{\delta_r}, G]$ and $Z_{\delta_r}/[Z_{\delta_{r-1}}, G]$, respectively. Then the Wedderburn decomposition of $\mathbb{Q}G$ is given by
	 	\[
	 	\mathbb{Q}G \cong \mathbb{Q}(G/G')  
	 	\;\;\bigoplus_{r=1}^n \;\;\bigoplus_{\substack{d_r \mid m_r \\ d_r \nmid m_r'}} 
	 	a_{d_r}\, M_{p^{\delta_r}}\!\left(\mathbb{Q}(\zeta_{d_r})\right)
	 	\;\;\bigoplus_{r=1}^n \;\;\bigoplus_{\substack{d_r \mid m_r \\ d_r \mid m_r'}} 
	 	\bigl(a_{d_r} - a'_{d_r}\bigr)\, M_{p^{\delta_r}}\!\left(\mathbb{Q}(\zeta_{d_r})\right),
	 	\]
	 	where $a_{d_r}$ and $a'_{d_r}$ denote the number of cyclic subgroups of order $d_r$ in $Z_{\delta_r}/[Z_{\delta_r}, G]$ and $Z_{\delta_r}/[Z_{\delta_{r-1}}, G]$, respectively.
	 \end{theorem}
	
     Note that a VZ-group has nilpotency class~$2$ (see~\cite{FM}). Furthermore, any nilpotent group of class~$2$ is a GVZ-group (see~\cite[Theorem~2.31]{I}). Hence, any nested group of class~$2$ is necessarily a nested GVZ-group. In particular, all two-generator $p$-groups of class~$2$ are nested GVZ $p$-groups (see~\cite{Nenciu2generators}). However, a two-generator $p$-group of class~$2$ need not be a VZ $p$-group. For example, the groups $\mathrm{SmallGroup}(729,24)$, $\mathrm{SmallGroup}(729,25)$, and $\mathrm{SmallGroup}(729,60)$ in the {\sf GAP} library are two-generator $3$-groups of class~$2$ that are nested GVZ but not VZ (see Example~\ref{example:2generator5}). In this article, using Theorem~\ref{thm:WedderburnGVZ}, we derive explicit combinatorial formulas for the Wedderburn decomposition of rational group algebras of two-generator $p$-groups, where $p$ is an odd prime. 
     
     Moreover, there exist nested GVZ $p$-groups of arbitrarily large nilpotency class. For every $n \geq 1$, Nenciu~\cite{NenciuGVZ2} constructed a family of nested GVZ $p$-groups of order $p^{2n+1}$, exponent~$p$, and class~$n+1$, where $p > n+1$ is prime. Similarly, for every $n \geq 1$, Lewis~\cite{LewisGVZ2} constructed a family of nested GVZ $p$-groups of order $p^{2n+1}$, exponent~$p^{\,n+1}$, and class~$n+1$, where $p$ is an odd prime. In this article, using Theorem~\ref{thm:WedderburnGVZ}, we also derive explicit combinatorial formulas for the Wedderburn decomposition of rational group algebras of these families of nested GVZ $p$-groups, where $p$ is an odd prime. 
     
     Further, we prove Theorem~\ref{thm:isoGVZ}, which shows that for a group, the property of being a GVZ-group (respectively, a nested GVZ-group) is preserved under the notion of \emph{isoclinism}. 
     
     \begin{theorem}\label{thm:isoGVZ}
     	Let $G$ and $H$ be two finite isoclinic groups. If $G$ is a GVZ-group (respectively, a nested GVZ-group), then $H$ is also a GVZ-group (respectively, a nested GVZ-group).
     \end{theorem}
     
     Using Theorem~\ref{thm:isoGVZ}, one can check that there are several isoclinic families of nested GVZ $p$-groups. In this article, we completely classify all nested GVZ $p$-groups of order at most $p^5$ (see Corollary~\ref{cor:isoGVZp^5}), and derive the Wedderburn decomposition of rational group algebras of these groups.
     
     \begin{corollary}\label{cor:isoGVZp^5}
     	Let $G$ be a non-abelian group of order $p^5$, where $p$ is an odd prime. Then $G$ is a nested GVZ-group if and only if $G \in \Phi_{2} \cup \Phi_{5} \cup \Phi_{7} \cup \Phi_{8}$.
     \end{corollary}
     \noindent Moreover, a non-abelian $p$-group $G$ of order at most $p^{5}$ is a nested GVZ-group that is not a VZ-group if and only if $|G| = p^{5}$ and $G$ belongs to either the isoclinic family $\Phi_{7}$ or the isoclinic family $\Phi_{8}$.
     
     Finally, we present a concise analysis of the primitive central idempotents and their associated simple components in the Wedderburn decomposition of rational group algebras of GVZ $p$-groups. In particular, we prove Theorem~\ref{thm:pciGVZ}.

	\begin{theorem}\label{thm:pciGVZ}
		Let $G$ be a finite nested GVZ $p$-group, where $p$ is an odd prime. Let $\chi \in \nl(G)$ with $N = \ker(\chi)$. Then we have the following.
	    \begin{enumerate}
		   \item $e_{\mathbb{Q}}(\chi)= \epsilon(Z(\chi), N)$. 
				
		   \item $\mathbb{Q}G\epsilon(Z(\chi), N) \cong M_{|G/Z(\chi)|^{1/2}}(\mathbb{Q}(\zeta_{|Z(\chi)/N|}))$.
		\end{enumerate} 
	\end{theorem}
	
	The structure of the article is as follows. In Section~\ref{sec:preliminaries}, we introduce notation, mostly following standard conventions, together with preliminary results that will be used throughout the paper. Section~\ref{sec:RationalGroupAlgebras} focuses on rational group algebras of nested GVZ $p$-groups, where $p$ is an odd prime, and contains the proof of Theorem~\ref{thm:WedderburnGVZ}. In Section~\ref{sec:2generator}, we discuss two-generator $p$-groups and derive explicit combinatorial formulas for the Wedderburn decomposition of their rational group algebras. Section~\ref{sec:GVZarbitrary} is devoted to the Wedderburn decomposition of rational group algebras corresponding to two families of nested GVZ $p$-groups of arbitrarily large nilpotency class. Section~\ref{sec:isoGVZ} contains the proofs of Theorem~\ref{thm:isoGVZ} and Corollary~\ref{cor:isoGVZp^5}. In this section, we also classify all nested GVZ $p$-groups of order at most $p^5$ and compute the Wedderburn decomposition of their rational group algebras. Finally, Section~\ref{sec:pci} examines the structure of primitive central idempotents in rational group algebras of nested GVZ $p$-groups and provides the proof of Theorem~\ref{thm:pciGVZ}.

	\section{Notation and Preliminaries}\label{sec:preliminaries}
	
	\subsection{Notation} We adopt the following notation, consistent with standard conventions. Unless otherwise specified, $p$ denotes an odd prime. For a finite group $G$, the notation introduced below will be used throughout.
	
	\begin{longtable}{cl}
		$[g, h]$ & $g^{-1}h^{-1}gh$ for $g, h \in G$\\
		$[g, G]$ & $\langle [g, x] : x\in G\rangle$ for $g \in G$\\
		$[N, G]$ & $\langle [n, g] : n \in N, g\in G\rangle$ for $N \trianglelefteq G$\\
		$G'$ & $[G, G]$, i.e., the commutator subgroup of $G$\\
		$\cl_G(g)$ & the conjugacy class of $g\in G$\\
		$\nil(G)$ & the nilpotency class of $G$\\
		$\Syl_p(G)$ &  the Sylow $p$-subgroup of $G$\\
		$|S|$ & the cardinality of a set $S$\\
		$|z|$ & the absolute value of $z \in \mathbb{C}$\\
		$\Irr(G)$ & the set of irreducible complex characters of $G$\\
		$\lin(G)$ & $\{\chi \in \Irr(G) : \chi(1)=1\}$\\
		$\nl(G)$ & $\{\chi \in \Irr(G) : \chi(1) \neq 1\}$\\
		$\Irr_m(G)$ & $\{\chi \in \Irr(G) : \chi(1)=m\}$\\
		$\cd(G)$ & $\{ \chi(1) : \chi \in \Irr(G) \}$\\
		$\mathbb{F}(\chi)$ & the field obtained by adjoining the values $\{\chi(g) : g\in G\}$ to the field $\mathbb{F}$ for some $\chi\in \Irr(G)$\\
		$m_\mathbb{Q}(\chi)$ & the Schur index of $\chi \in \Irr(G)$ over $\mathbb{Q}$\\
		$\Omega(\chi)$ & $m_{\mathbb{Q}}(\chi)\sum_{\sigma \in \gal(\mathbb{Q}(\chi) / \mathbb{Q})}^{}\chi^{\sigma}$ for $\chi \in \Irr(G)$\\
		$\ker(\chi)$ & $\{g \in G: \chi(g)=\chi(1)\}$ for $\chi \in \Irr(G)$\\
		$Z(\chi)$ & $\{g \in G: |\chi(g)|=\chi(1)\}$ for $\chi \in \Irr(G)$\\
		$\Irr(G|N)$ & $\{\chi \in \Irr(G) : N \nsubseteq \ker(\chi)\}$, where $N \trianglelefteq G$\\
		$\chi \downarrow_H$ & the restriction of a character $\chi$ of $G$ on $H$, where $H \leq G$\\
		$\mathbb{F}G$ & the group ring (algebra) of $G$ with coefficients in $\mathbb{F}$\\
		$M_{n}(D)$ & a full matrix ring of order $n$ over the skewfield $D$\\
		$Z(B)$ & the center of an algebraic structure $B$\\
		$\zeta_m$ & an $m$-th primitive root of unity\\
		$C_n$ & the cyclic group of order $n$
	\end{longtable}
	
	\subsection{Preliminaries}
	
	Here, we introduce some fundamental concepts and results that will be used repeatedly throughout this article. For a finite abelian $p$-group $G$, Yeh~\cite{Cyclic subgroups1} obtained an explicit formula for the number of subgroups of a prescribed type, where $p$ is any prime. More recently, an alternative expression was derived in~\cite{Cyclic subgroups2}, using a method based on certain matrices associated with the invariant factor decomposition of $G$. The following lemma provides the precise counting formula for cyclic subgroups.
	
	\begin{lemma}\cite[Theorem~4.3]{Cyclic subgroups2}\label{lemma:CyclicCounting}
		Let $G$ be a finite abelian $p$-group such that $G \cong C_{p^{\alpha_1}} \times C_{p^{\alpha_2}} \times \cdots \times C_{p^{\alpha_k}}$,  
		with $\alpha_1 \leq \alpha_2 \leq \cdots \leq \alpha_k$, where $p$ is a prime. For each $1 \leq \alpha \leq \alpha_k$, denote by $\mathcal{G}_p^k(\alpha)$ the number of cyclic subgroups of order $p^\alpha$ in $G$. Then
		\[
		\mathcal{G}_p^k(\alpha) = \frac{p \, h_p^{k-1}(\alpha) - h_p^{k-1}(\alpha-1)}{p-1},
		\]
		where $h_p^{k-1}(\alpha) = p^{(k-j)\alpha + \alpha_1 + \cdots + \alpha_{j-1}}$ whenever $\alpha_{j-1} \leq \alpha \leq \alpha_j$, for $j \in \{1, 2, \ldots, k\}$ with the convention $\alpha_0 = 0$.
	\end{lemma}
	
	We next recall some essential results concerning characters and representations that will be required later.
	
	\begin{lemma}\cite[Corollary~10.14]{I}\label{lemma:schurindexpgroup}
		Let $G$ be a $p$-group, where $p$ is an odd prime. Then 
		\[
		m_\mathbb{Q}(\chi) = 1 \quad \text{for all } \chi \in \Irr(G).
		\]
	\end{lemma}
	
	Let $G$ be a finite group. Define an equivalence relation on $\Irr(G)$ by Galois conjugacy over $\mathbb{Q}$. Two irreducible characters $\chi, \psi \in \Irr(G)$ are called \emph{Galois conjugates} over $\mathbb{Q}$ if $\mathbb{Q}(\chi) = \mathbb{Q}(\psi)$ and there exists $\sigma \in \gal(\mathbb{Q}(\chi)/\mathbb{Q})$ such that $\chi^\sigma = \psi$.
	
	\begin{lemma}\cite[Lemma~9.17]{I}\label{SC}
		Let $G$ be a finite group and $\chi \in \Irr(G)$. Denote by $E(\chi)$ the Galois conjugacy class of $\chi$ over $\mathbb{Q}$. Then
		\[
		|E(\chi)| = [\mathbb{Q}(\chi) : \mathbb{Q}].
		\]
	\end{lemma}
	
	It follows that distinct Galois conjugacy classes correspond to distinct irreducible rational representations of $G$. Reiner~\cite[Theorem~3]{IR} further described the structure of the simple components in the Wedderburn decomposition of the rational group algebra $\mathbb{Q}G$ associated with these rational representations. We conclude this section by citing his result.
	
	\begin{lemma}\cite[Theorem~3]{IR}\label{Reiner}
		Let $\mathbb{K}$ be a field of characteristic zero and let $\mathbb{K}^*$ denote its algebraic closure. Suppose $T$ is an irreducible $\mathbb{K}$-representation of $G$, extended linearly to a $\mathbb{K}$-representation of $\mathbb{K}G$. Set
		\[
		A = \{ T(x) : x \in \mathbb{K}G \}.
		\]
		Then $A$ is a simple algebra over $\mathbb{K}$, and may be written as $A = M_n(D)$, where $D$ is a division ring. Moreover,
		\[
		Z(D) \cong \mathbb{K}(\chi_i) \quad \text{and} \quad [D : Z(D)] = \big(m_\mathbb{K}(\chi_i)\big)^2 \quad (1 \leq i \leq k),
		\]
		where $U_i$ are irreducible $\mathbb{K}^*$-representations of $G$ affording the characters $\chi_i$, 
		\[
		T = m_\mathbb{K}(\chi_i)\bigoplus_{i=1}^k U_i,
		\]
		with $k = [\mathbb{K}(\chi_i) : \mathbb{K}]$, and $m_\mathbb{K}(\chi_i)$ denoting the Schur index of $\chi_i$ over $\mathbb{K}$.
	\end{lemma}

	\section{Rational group algebra}\label{sec:RationalGroupAlgebras}
	In this section, we establish the proof of Theorem~\ref{thm:WedderburnGVZ}.  
	Let $G$ be a nested GVZ-group. Then for every $\chi \in \Irr(G)$, we have  $\chi(1)^2 = |G/Z(\chi)|$ (see \cite[Corollary~2.30]{I}). Moreover, if $\chi, \psi \in \Irr(G)$ satisfy $\chi(1)=\psi(1)$, then $Z(\chi)=Z(\psi)$. Let  
	\[
	\cd(G) = \{p^{\delta_i} : 0 \leq i \leq n, \; 0=\delta_0 < \delta_1 < \cdots < \delta_n\},
	\]  
	and define $Z_{\delta_i}:=Z(\chi)$ for some $\chi \in \Irr_{p^{\delta_i}}(G)$, where $\Irr_{p^{\delta_i}}(G) = \{\chi \in \Irr(G) : \chi(1)=p^{\delta_i}\}$. Note that for each $\delta_r \in \{\delta_1, \dots, \delta_n\}$, the quotient group $Z_{\delta_r}/[Z_{\delta_r}, G]$ is abelian. With this notation, we recall the following result.
	
	\begin{lemma}\cite[Theorem~3.8]{NenciuGVZ}\label{lemma:GVZcharacter}
		Let $G$ be a nested GVZ-group with  $\cd(G)=\{p^{\delta_i} : 0 \leq i \leq n, \; 0=\delta_0 < \delta_1 < \cdots < \delta_n\}$. Then for each $r \in \{1, \dots, n\}$, there is a bijection between the sets  
		\[
		\Irr\big(Z_{\delta_r}/[Z_{\delta_r}, G]  \mid [Z_{\delta_{r-1}}, G]/[Z_{\delta_r}, G]\big)
		\quad \text{and} \quad 
		\Irr_{p^{\delta_r}}(G).
		\]  
		Moreover, for $\bar{\mu} \in \Irr(Z_{\delta_r}/[Z_{\delta_r}, G] \mid [Z_{\delta_{r-1}}, G]/[Z_{\delta_r}, G])$, the associated character $\chi_\mu \in \Irr_{p^{\delta_r}}(G)$ is given by
		\begin{equation}\label{GVZcharacter}
			\chi_\mu(g) =
			\begin{cases}
				p^{\delta_r}\mu(g) & \text{if } g \in Z_{\delta_r}, \\
				0 & \text{otherwise},
			\end{cases}
		\end{equation}
		where $\mu$ denotes the lift of $\bar{\mu}$ to $Z_{\delta_r}$.
	\end{lemma}
	
	Let $G$ be a finite group, and let $\chi, \psi \in \Irr(G)$ be Galois conjugate over $\mathbb{Q}$. Then necessarily $\ker(\chi)=\ker(\psi)$. The converse is true when $\chi, \psi \in \lin(G)$. However, for general $\chi, \psi \in \nl(G)$ with $\ker(\chi)=\ker(\psi)$, it need not follow that they are Galois conjugates. In the case of nested GVZ-groups, the converse does hold for nonlinear irreducible complex characters, as we show below.
	
	\begin{lemma}\label{lemma:galoisGVZ}
		Let $G$ be a nested GVZ-group with  $\cd(G)=\{p^{\delta_i} : 0 \leq i \leq n, \; 0=\delta_0 < \delta_1 < \cdots < \delta_n\}$. For $\chi, \psi \in \Irr_{p^{\delta_r}}(G)$ with some $r \in \{1, \dots, n\}$, we have the following.
		\[
		\chi \text{ and } \psi \text{ are Galois conjugates over } \mathbb{Q}
		 \Longleftrightarrow  \ker(\chi)=\ker(\psi).
		\]
	\end{lemma}
	
	\begin{proof}
		If $\chi$ and $\psi$ are Galois conjugates over $\mathbb{Q}$, then $\ker(\chi)=\ker(\psi)$ is immediate. Conversely, suppose $\chi, \psi \in \Irr_{p^{\delta_r}}(G)$ with $\ker(\chi)=\ker(\psi)$.  
		By definition, $Z_{\delta_i} := Z(\phi)$ for some $\phi \in \Irr_{p^{\delta_i}}(G)$. From \eqref{GVZcharacter}, there exist $\mu, \nu \in \Irr(Z_{\delta_r})$ such that  
		\[
		\chi\!\downarrow_{Z_{\delta_r}} = p^{\delta_r}\mu, \quad 
		\psi\!\downarrow_{Z_{\delta_r}} = p^{\delta_r}\nu,
		\]  
		where $\mu$ and $\nu$ are the lifts of some $\bar{\mu} \in \Irr(Z_{\delta_r}/[Z_{\delta_r}, G] \mid [Z_{\delta_{r-1}}, G]/[Z_{\delta_r}, G])$ to $Z_{\delta_r}$ and $\bar{\nu} \in \Irr(Z_{\delta_r}/[Z_{\delta_r}, G] \mid [Z_{\delta_{r-1}}, G]/[Z_{\delta_r}, G])$ to $Z_{\delta_r}$, respectively.  
		Since $\ker(\chi)=\ker(\mu)$ and $\ker(\psi)=\ker(\nu)$, we deduce that $\mu$ and $\nu$ are Galois conjugates. Hence, $\mathbb{Q}(\mu)=\mathbb{Q}(\nu)$. Therefore, there exists $\sigma \in \gal(\mathbb{Q}(\mu)/\mathbb{Q})$ with $\mu^\sigma=\nu$. As $\mathbb{Q}(\chi)=\mathbb{Q}(\mu)$ and $\mathbb{Q}(\psi)=\mathbb{Q}(\nu)$, it follows that $\chi^\sigma=\psi$. Thus, $\chi$ and $\psi$ are Galois conjugates over $\mathbb{Q}$. This completes the proof of Lemma~\ref{lemma:galoisGVZ}.
	\end{proof}
	
	Now, for $\chi, \psi \in \Irr(G)$, we say that $\chi$ and $\psi$ are \emph{equivalent} if $\ker(\chi)=\ker(\psi)$.
	
	\begin{lemma}\cite[Lemma 1]{Ayoub}\label{lemma:Ayoub}
		Let $G$ be a finite abelian group of exponent $m$, and let $d \mid m$. If $a_d$ denotes the number of cyclic subgroups of $G$ of order $d$, then the number of inequivalent characters $\chi$ with $\mathbb{Q}(\chi)=\mathbb{Q}(\zeta_d)$ is $a_d$.
	\end{lemma}
	
	Analogous to Lemma~\ref{lemma:Ayoub}, and under the above hypotheses, we now determine the number of inequivalent irreducible characters of a nested GVZ-group corresponding to a given cyclotomic field.
	
	\begin{lemma}\label{lemma:GVZgaloisconjugates}
		Let $G$ be a nested GVZ-group with $\cd(G)=\{p^{\delta_i} : 0 \leq i \leq n, \; 0=\delta_0 < \delta_1 < \cdots < \delta_n\}$. Define $Z_{\delta_i}:=Z(\chi)$ for some $\chi \in \Irr_{p^{\delta_i}}(G)$. For each $r \in \{1, \dots, n\}$, let $d_r$ be a divisor of $\exp(Z_{\delta_r}/[Z_{\delta_r}, G])$. Let $a_{d_r}$ and $a'_{d_r}$ denote the number of cyclic subgroups of order $d_r$ in $Z_{\delta_r}/[Z_{\delta_r}, G]$ and $Z_{\delta_r}/[Z_{\delta_{r-1}}, G]$, respectively. Finally, let $m_{d_r}$ be the number of inequivalent characters $\chi \in \Irr_{p^{\delta_r}}(G)$ with $\mathbb{Q}(\chi)=\mathbb{Q}(\zeta_{d_r})$. Then we have the following.
		\begin{enumerate}
			\item If $d_r \mid \exp(Z_{\delta_r}/[Z_{\delta_r}, G])$ but $d_r \nmid \exp(Z_{\delta_r}/[Z_{\delta_{r-1}}, G])$, then $m_{d_r}=a_{d_r}$.
			\item If $d_r$ divides both $\exp(Z_{\delta_r}/[Z_{\delta_r}, G])$ and $\exp(Z_{\delta_r}/[Z_{\delta_{r-1}}, G])$, then $m_{d_r}=a_{d_r}-a'_{d_r}$.
		\end{enumerate}
	\end{lemma}
	
	\begin{proof}
		Let $\chi \in \Irr_{p^{\delta_r}}(G)$. Then $\chi=\chi_\mu$ for some $\mu \in \Irr(Z_{\delta_r})$, as in \eqref{GVZcharacter}, where $\mu$ is the lift of $\bar{\mu}\in \Irr(Z_{\delta_r}/[Z_{\delta_r}, G] \mid [Z_{\delta_{r-1}}, G]/[Z_{\delta_r}, G])$ to $Z_{\delta_r}$. We observe that $\ker(\chi_\mu)=\ker(\mu)$ and $\mathbb{Q}(\chi_\mu)=\mathbb{Q}(\mu)$. By Lemma~\ref{lemma:GVZcharacter}, there is a bijection between the sets $\Irr(Z_{\delta_r}/[Z_{\delta_r}, G] \mid [Z_{\delta_{r-1}}, G]/[Z_{\delta_r}, G])$ and $\Irr_{p^{\delta_r}}(G)$. Consequently, by applying Lemma \ref{lemma:Ayoub}, the results follow. This completes the proof of Lemma \ref{lemma:GVZgaloisconjugates}.
	\end{proof}
	
	We are now ready to prove Theorem~\ref{thm:WedderburnGVZ}.

	\begin{proof}[Proof of Theorem~\ref{thm:WedderburnGVZ}]
			Let $G$ be a finite nested GVZ $p$-group, where $p$ is an odd prime, and let $\chi \in \Irr(G)$. Suppose $\rho$ is an irreducible $\mathbb{Q}$-representation of $G$ affording the character $\Omega(\chi)$. Denote by $A_\mathbb{Q}(\chi)$ the simple component in the Wedderburn decomposition of $\mathbb{Q}G$ corresponding to $\rho$, so that $A_\mathbb{Q}(\chi) \cong M_q(D)$ for some $q \in \mathbb{N}$ and a division algebra $D$. By Lemma~\ref{lemma:schurindexpgroup}, we have $m_\mathbb{Q}(\chi) = 1$. Moreover, Lemma~\ref{Reiner} shows that $[D:Z(D)] = m_\mathbb{Q}(\chi)^2$ and $Z(D) = \mathbb{Q}(\chi)$. Hence, $D = Z(D) = \mathbb{Q}(\chi)$. Next, consider
			\[
			\rho \cong \bigoplus_{i=1}^l \rho_i,
			\]
			where $l = [\mathbb{Q}(\chi): \mathbb{Q}]$ and each $\rho_i$ is an irreducible complex representation of $G$ affording the character $\chi^{\sigma_i}$ for some $\sigma_i \in \mathrm{Gal}(\mathbb{Q}(\chi)/\mathbb{Q})$. Since $m_\mathbb{Q}(\chi) = 1$, it follows from \cite[Theorem 3.3.1]{JR} that $q = \chi(1)$.  
			
			Note that $\lin(G) \cong G/G'$. Therefore, the simple components in the Wedderburn decomposition of $\mathbb{Q}G$ corresponding to all inequivalent irreducible $\mathbb{Q}$-representations of $G$ affording the character $\Omega(\chi)$ with $\chi \in \lin(G)$ are precisely isomorphic to $\mathbb{Q}(G/G')$.  
			
			Furthermore, we have
			\[
			\cd(G) = \{p^{\delta_i} : 0 \leq i \leq n, \; 0 = \delta_0 < \delta_1 < \cdots < \delta_n\},
			\]
			and for each $i$, let $Z_{\delta_i} := Z(\chi)$ for some $\chi \in \Irr_{p^{\delta_i}}(G)$. Note that $Z_{\delta_i}/[Z_{\delta_i}, G]$ is abelian for every $i \in \{0,1,\dots,n\}$.  
			
			Fix $r \in \{1,2,\dots,n\}$, and let $\rho$ be an irreducible $\mathbb{Q}$-representation of $G$ affording the character $\Omega(\chi_\mu)$, where $\chi_\mu \in \Irr_{p^{\delta_r}}(G)$ (given in \eqref{GVZcharacter}). Here, $\chi_\mu(1) = p^{\delta_r}$ and $\mathbb{Q}(\chi_\mu) = \mathbb{Q}(\mu)$. Therefore, by the above discussion, we have
			\[
			A_\mathbb{Q}(\chi_\mu) \cong M_{p^{\delta_r}}(\mathbb{Q}(\mu)).
			\]
			Moreover, $\mathbb{Q}(\chi_\mu) = \mathbb{Q}(\mu) = \mathbb{Q}(\zeta_{d_r})$ for some $d_r$ dividing $\exp(Z_{\delta_r}/[Z_{\delta_r}, G])$.  
			
			We now distinguish two cases.  
			
			\noindent\textbf{Case 1.} If $d_r \mid \exp(Z_{\delta_r}/[Z_{\delta_r}, G])$ but $d_r \nmid \exp(Z_{\delta_r}/[Z_{\delta_{r-1}}, G])$, then by Lemmas~\ref{lemma:galoisGVZ} and~\ref{lemma:GVZgaloisconjugates}(1), the number of inequivalent irreducible $\mathbb{Q}$-representations of $G$ affording the character $\Omega(\chi_\mu)$ for some character $\chi_\mu \in \Irr_{p^{\delta_r}}(G)$ with $\mathbb{Q}(\chi_\mu) = \mathbb{Q}(\zeta_{d_r})$ equals $a_{d_r}$, where $a_{d_r}$ is the number of cyclic subgroups of order $d_r$ in $Z_{\delta_r}/[Z_{\delta_r}, G]$.  
			
			\noindent\textbf{Case 2.} If $d_r \mid \exp(Z_{\delta_r}/[Z_{\delta_r}, G])$ and $d_r \mid \exp(Z_{\delta_r}/[Z_{\delta_{r-1}}, G])$, then by Lemmas~\ref{lemma:galoisGVZ} and~\ref{lemma:GVZgaloisconjugates}(2), the number of inequivalent irreducible $\mathbb{Q}$-representations of $G$ affording the character $\Omega(\chi_\mu)$ for some character $\chi_\mu \in \Irr_{p^{\delta_r}}(G)$ with $\mathbb{Q}(\chi_\mu) = \mathbb{Q}(\zeta_{d_r})$ equals $a_{d_r} - a_{d_r}'$, where $a_{d_r}$ and $a_{d_r}'$ denote the numbers of cyclic subgroups of order $d_r$ in $Z_{\delta_r}/[Z_{\delta_r}, G]$ and $Z_{\delta_r}/[Z_{\delta_{r-1}}, G]$, respectively.  
			
			Now, let $m_r$ and $m_r'$ be the exponents of $Z_{\delta_r}/[Z_{\delta_r}, G]$ and $Z_{\delta_r}/[Z_{\delta_{r-1}}, G]$, respectively. By combining the two cases, the simple components of $\mathbb{Q}G$ corresponding to all inequivalent irreducible $\mathbb{Q}$-representations of $G$ affording the character $\Omega(\chi_\mu)$, where $\chi_\mu \in \Irr_{p^{\delta_r}}(G)$, contribute
            $$\bigoplus_{\substack{d_r \mid m_r \\ d_r \nmid m_r'}} 
            a_{d_r}\, M_{p^{\delta_r}}\!\left(\mathbb{Q}(\zeta_{d_r})\right) \;\;\bigoplus_{\substack{d_r \mid m_r \\ d_r \mid m_r'}} 
            \bigl(a_{d_r} - a'_{d_r}\bigr)\, M_{p^{\delta_r}}\!\left(\mathbb{Q}(\zeta_{d_r})\right)$$
			to the Wedderburn decomposition of $\mathbb{Q}G$.  
			
			Finally, by collecting the simple components corresponding to all inequivalent irreducible $\mathbb{Q}$-representations of $G$ affording the character $\Omega(\chi)$ with $\chi \in \Irr_{p^{\delta_i}}(G)$ for all $i \in \{0,1,\dots,n\}$, the result follows. This completes the proof of Theorem~\ref{thm:WedderburnGVZ}.
	\end{proof}

	\section{Two-generator $p$-groups of nilpotency class $2$}\label{sec:2generator}
	
	The classification of two-generator $p$-groups of class $2$ is given in \cite{2generatorp-group}. In this section, we begin by introducing some notation and terminology that will allow us to restate the main result of \cite{2generatorp-group}.  
	Let $p$ be a prime. For a given integer $n>2$, define a set of $5$-tuples
	$$\tau_n=\{(\alpha, \beta, \gamma; \rho, \sigma) ~:~ \alpha \geq \beta\geq\gamma\geq 1,~ \alpha+\beta+\gamma=n,~ 0\leq \rho \leq \gamma,~ 0\leq\sigma \leq \gamma \}.$$ 
	For each $(\alpha, \beta, \gamma; \rho, \sigma)\in \tau_n$, consider the group
	\begin{equation}\label{present:2generator p-group}
		G=G_{(\alpha, \beta, \gamma; \rho, \sigma)}=\langle a, b~:~[a, b]^{p^\gamma}=[a, b, a]=[a, b, b]=1, ~a^{p^\alpha}=[a, b]^{p^\rho}, ~b^{p^\beta}=[a, b]^{p^\sigma} \rangle.
	\end{equation}
	It is clear from the presentation that $G=G_{(\alpha, \beta, \gamma; \rho, \sigma)}$ is a two-generator $p$-group of class $2$ with order $p^n$. Moreover, the derived subgroup is 
	 $$G'=\langle [a, b] \rangle \cong C_{p^\gamma}.$$ 
	 and the center is 
	 $$Z(G) =\langle a^{p^\gamma}, b^{p^\gamma}, [a, b] \rangle.$$
	 
\noindent 	 When $\rho \leq \sigma$, the subgroup $\langle a^{p^\gamma}\rangle \cong C_{p^{\alpha-\rho}}$ forms a cyclic direct factor $Z(G)$, and 
	 $$Z(G)/{\langle a^{p^\gamma}\rangle} =\big \langle b^{p^\gamma}\langle a^{p^\gamma}\rangle,~ [a, b]\langle a^{p^\gamma}\rangle \big\rangle \cong C_{p^{\beta-\gamma}} \times C_{p^{\rho}}.$$ 
	 Similarly, if $\sigma < \rho$, then $\langle b^{p^\gamma}\rangle \cong C_{p^{\beta-\sigma}}$ is a cyclic direct factor of $Z(G)$, and 
	 $$Z(G)/{\langle b^{p^\gamma}\rangle} =\big \langle a^{p^\gamma}\langle b^{p^\gamma}\rangle,~ [a, b]\langle b^{p^\gamma}\rangle \big\rangle \cong C_{p^{\alpha-\gamma}} \times C_{p^{\sigma}}.$$
	 Thus, we obtain 
	$$Z(G) \cong \begin{cases}
		C_{p^{\alpha-\rho}} \times C_{p^{\beta-\gamma}} \times C_{p^{\rho}} & \text{if } \rho \leq \sigma,\\
		C_{p^{\alpha-\gamma}} \times C_{p^{\beta-\sigma}} \times C_{p^{\sigma}} & \text{if } \sigma < \rho.
	\end{cases}$$
	
   Next, we introduce the following subsets of $\tau_n$:
	\begin{align*}
		\tau_{n_1}&=\{(\alpha, \beta, \gamma; \rho, \gamma) \in \tau_n ~:~ \alpha > \beta\geq\gamma\geq \rho \geq 0 \},\\
		\tau_{n_2}&=\{(\alpha, \beta, \gamma; \gamma, \sigma) \in \tau_n ~:~ \alpha > \beta\geq\gamma > \sigma \geq 0 \},\\
		\tau_{n_3}&=\{(\alpha, \beta, \gamma; \rho, \sigma) \in \tau_n ~:~ \alpha > \beta\geq\gamma ~\text{and}~ \min(\gamma, ~\sigma+\alpha-\beta)>\rho>\sigma \geq 0 \},\\
		\tau_{n_4}&=\{(\alpha, \alpha, \gamma; \rho, \gamma) \in \tau_n ~:~ \alpha > \gamma \geq \rho \geq 0 \},\\
		\tau_{n_5}&=\{(\gamma, \gamma, \gamma; \rho, \gamma) \in \tau_n ~:~ 0 \leq \rho \leq \gamma \}.
	\end{align*}
	
\noindent For an odd prime $p$, a $5$-tuple $(\alpha, \beta, \gamma; \rho, \sigma)\in \tau_n$ is called \emph{$p$-good 5-tuple} if it belongs to $\tau_{n_1}\cup\tau_{n_2}\cup\tau_{n_3}\cup\tau_{n_4}\cup\tau_{n_5}$.
	\begin{lemma}\cite[Theorem 1.1]{2generatorp-group}\label{lemma:classification2generator}
		Let $G$ be a two-generator $p$-group of class $2$ and order $p^n$, with $p$ an odd prime. Then there exists a unique $p$-good $5$-tuple $(\alpha, \beta, \gamma; \rho, \sigma)\in \tau_{n_1}\cup \tau_{n_2}\cup\tau_{n_3}\cup\tau_{n_4}\cup\tau_{n_5}$ such that $G \cong G_{(\alpha, \beta, \gamma; \rho, \sigma)}$.
	\end{lemma}
	
 Nenciu~\cite{Nenciu2generators} showed that every two-generator $p$-group of class $2$ is a nested GVZ $p$-group. The next lemma describes some additional properties of such groups, which is a summary of \cite[Theorems 3.1, 3.3 and 4.1]{Nenciu2generators}.
	\begin{lemma}\label{lemma:property2generator}
		Let $G= G_{(\alpha, \beta, \gamma; \rho, \sigma)}$ as defined in~\eqref{present:2generator p-group} be a two-generator $p$-group of class $2$ and order $p^n$, where $p$ is an odd prime. Then we have the following.
		\begin{enumerate}
			\item $\cd(G)=\{1, p, \cdots, p^\gamma\}$.
					\item If $\chi \in \Irr(G)$ with $\chi(1)=p^\delta$ for some $\delta \in \{0, 1, \dots, \gamma\}$, then $Z(\chi)=\langle a^{p^\delta},~ b^{p^\delta},~ [a, b]\rangle$ and $[Z(\chi), G] = \langle {[a, b]}^{p^\delta} \rangle$.
		\end{enumerate}
	\end{lemma} 

	We are now ready to prove Theorem~\ref{thm:Wedderburn2generator}, which provides a combinatorial description of the Wedderburn decomposition of rational group algebras for the family of groups defined in~\eqref{present:2generator p-group}.
	\begin{theorem}\label{thm:Wedderburn2generator}
		Let $G= G_{(\alpha, \beta, \gamma; \rho, \sigma)}$ (defined in~\eqref{present:2generator p-group}) be a two-generator $p$-group of class $2$ and order $p^n$, where $p$ is an odd prime. Then
		$$\mathbb{Q}G \cong \mathbb{Q} \bigoplus_{\lambda=1}^\beta (p^\lambda+p^{\lambda-1})\mathbb{Q}(\zeta_{p^\lambda}) \bigoplus_{\lambda=\beta+1}^\alpha p^\beta \mathbb{Q}(\zeta_{p^\lambda}) \bigoplus_{\delta=1}^\gamma \bigoplus_{\lambda=1}^{m_\delta}(a_{p^\lambda}-a_{p^\lambda}') M_{p^{\delta}}(\mathbb{Q}(\zeta_{p^\lambda})),$$
		where $a_{p^\lambda}$ and $a_{p^\lambda}'$ are the number of cyclic subgroups of  order $p^\lambda$ of $\big \langle a^{p^\delta},~ b^{p^\delta},~ [a, b]\big\rangle\big/\big\langle {[a, b]}^{p^\delta} \big\rangle$ and $\big \langle a^{p^\delta},~ b^{p^\delta},~ [a, b]\big\rangle\big/\big\langle {[a, b]}^{p^{\delta-1}} \big\rangle$, respectively, and $p^{m_\delta}$ denotes the exponent of $\big \langle a^{p^\delta},~ b^{p^\delta},~ [a, b]\big\rangle\big/\big\langle {[a, b]}^{p^\delta} \big\rangle$.
	\end{theorem}
	\begin{proof}
		As $(\alpha, \beta, \gamma; \rho, \gamma) \in \tau_n$, we have $\alpha \geq \beta \geq \gamma \geq 1$. From \eqref{present:2generator p-group}, we have
		\begin{equation*}
			G=G_{(\alpha, \beta, \gamma; \rho, \sigma)}=\langle a, b~:~[a, b]^{p^\gamma}=[a, b, a]=[a, b, b]=1, ~a^{p^\alpha}=[a, b]^{p^\rho}, ~b^{p^\beta}=[a, b]^{p^\sigma} \rangle.
		\end{equation*}
		Note that $G'=\big \langle [a, b] \big \rangle \cong C_{p^\gamma}$ and the quotient $G/G'=\big \langle aG', bG' \big \rangle \cong C_{p^\alpha} \times C_{p^\beta}$. For each $1 \leq \lambda \leq \alpha$, let $\mathcal{G/G'}_p(\lambda)$ denote the number of cyclic subgroups of order $p^\lambda$ in $G/G'$. By Lemma~\ref{lemma:CyclicCounting}, for $\lambda$ with $1 \leq \lambda \leq \beta$, we have $k=2$ and 
		\begin{align*}
			\mathcal{G/G'}_p(\lambda)=  \frac{p h_p^1(\lambda)- h_p^1(\lambda-1)}{p-1},
		\end{align*}
		where, $h_p^1(\lambda)=p^\lambda$ and $h_p^1(\lambda-1)=p^{\lambda-1}$. Hence, we get
		\begin{align*}
			\mathcal{G/G'}_p(\lambda)= \frac{p^{\lambda+1}-p^{\lambda-1}}{p-1}=p^{\lambda}+p^{\lambda-1}.
		\end{align*}
		Similarly, for $\lambda$ with $\beta+1 \leq \lambda \leq \alpha$, we have $k=2$ and 
		\begin{align*}
			\mathcal{G/G'}_p(\lambda)=  \frac{p h_p^1(\lambda)- h_p^1(\lambda-1)}{p-1} = \frac{p^{\beta+1}-p^{\beta}}{p-1}=p^{\beta}
		\end{align*}
	     as $h_p^1(\lambda)=h_p^1(\lambda-1)=p^{\beta}$ (see Lemma~\ref{lemma:CyclicCounting}). Hence, from Theorem~\ref{Perlis-walker}, we get
		$$\mathbb{Q}(G/G') \cong \mathbb{Q} \bigoplus_{\lambda=1}^\beta (p^\lambda+p^{\lambda-1})\mathbb{Q}(\zeta_{p^\lambda}) \bigoplus_{\lambda=\beta+1}^\alpha p^\beta \mathbb{Q}(\zeta_{p^\lambda}).$$	
			
		For each $\delta \in \{0,1,\dots,\gamma\}$, let $Z_\delta=Z(\chi)$ for $\chi \in \Irr(G)$ with $\chi(1)=p^\delta$. Then by Lemma~\ref{lemma:property2generator}, we get  
	     $$Z_\delta=\langle a^{p^\delta},~ b^{p^\delta},~ [a, b]\rangle \quad \text{and} \quad [Z_\delta, G] = \langle {[a, b]}^{p^\delta} \rangle.$$ 
	     Moreover, $\cd(G)=\{1, p, \dots, p^\gamma\}$. 
		Suppose $p^{m_\delta}$ is the exponent of $Z_\delta/[Z_\delta, G]$ for each $\delta \in \{1, 2, \cdots, \gamma\}$. We have 
		$$Z_\delta/[Z_{\delta-1}, G]= (Z_\delta/[Z_\delta, G])\big /([Z_{\delta -1}]/[Z_\delta, G]).$$
		This implies that $\exp(Z_\delta/[Z_{\delta-1}, G])$ divides $p^{m_\delta}$. Therefore, by applying Theorem~\ref{thm:WedderburnGVZ}, the result follows. This completes the proof of Theorem~\ref{thm:Wedderburn2generator}.
	\end{proof}

	\subsection{An explicit formula for $\mathbb{Q}G_{(\alpha,\beta,\gamma;\rho,\sigma)}$ with $(\alpha,\beta,\gamma;\rho,\sigma) \in \tau_{n_5}$} 
	
	In this subsection, we derive an explicit combinatorial expression for the Wedderburn decomposition of rational group algebras corresponding to a particular subclass of two-generator $p$-groups of class $2$. This serves as a concrete illustration of Theorem~\ref{thm:Wedderburn2generator}. Here, we restrict our attention to groups that are isomorphic to $G_{(\alpha, \beta, \gamma; \rho, \sigma)}$ with $(\alpha, \beta, \gamma; \rho, \sigma) \in \tau_{n_5}$. We begin with the following lemma.  
	\begin{lemma}\label{lemma:2generator5}
		Let $G= G_{(\alpha, \beta, \gamma; \rho, \sigma)}$ as defined in~\eqref{present:2generator p-group} be a two generator $p$-group of class two and of order $p^n$ for an odd prime $p$, such that $(\alpha, \beta, \gamma; \rho, \sigma) \in \tau_{n_5}$. Suppose $1 \leq \delta \leq \gamma$. Then we have the following.
		\begin{enumerate}
			\item $\big \langle a^{p^\delta},~ b^{p^\delta},~ [a, b]\big\rangle\big/\big\langle {[a, b]}^{p^\delta} \big\rangle \cong \begin{cases}
				C_{p^{\gamma-\rho}} \times C_{p^{\gamma-\delta}} \times C_{p^\rho} & \text{if } \rho < \delta<\gamma-\rho, \\
				C_{p^{\gamma-\delta}} \times C_{p^{\gamma-\delta}} \times C_{p^\delta} & \text{otherwise}.
			\end{cases}$
			
			\item $\big \langle a^{p^\delta},~ b^{p^\delta},~ [a, b]\big\rangle\big/\big\langle {[a, b]}^{p^{\delta-1}} \big\rangle \cong \begin{cases}
				C_{p^{\gamma-\rho-1}} \times C_{p^{\gamma-\delta}} \times C_{p^\rho} & \text{if } \rho < \delta<\gamma-\rho, \\
				C_{p^{\gamma-\delta}} \times C_{p^{\gamma-\delta}} \times C_{p^{\delta-1}} & \text{otherwise}.
			\end{cases}$
		\end{enumerate}
	\end{lemma}
	\begin{proof}
		Since $(\alpha, \beta, \gamma; \rho, \sigma) \in \tau_{n_5}$, we have $\alpha=\beta=\sigma=\gamma$ and $0\leq\rho \leq \gamma$. Thus, from \eqref{present:2generator p-group}, we have
		\begin{equation*}
			G=G_{(\gamma, \gamma, \gamma; \rho, \gamma)}=\langle a, b~:~[a, b]^{p^\gamma}=[a, b, a]=[a, b, b]=1, ~a^{p^\gamma}=[a, b]^{p^\rho}, ~b^{p^\gamma}=1 \rangle.
		\end{equation*}
		
		For $1 \leq \delta \leq \gamma$, we have   
		$$Z_\delta =\big \langle a^{p^\delta},~ b^{p^\delta},~ [a, b]\big\rangle; \quad [Z_\delta, G]=\big\langle {[a, b]}^{p^\delta} \big\rangle \quad \text{and} \quad [Z_{\delta-1}, G]=\big\langle {[a, b]}^{p^{\delta-1}} \big\rangle,$$
		where $Z_{\delta}=Z(\chi)$ for some $\chi\in \Irr_{p^{\delta}}(G)$ and $Z_{\delta-1}=Z(\psi)$ for some $\psi\in \Irr_{p^{\delta-1}}(G)$ (see Lemma \ref{lemma:property2generator}).
		\begin{enumerate}
			\item First note that  $\big\langle b^{p^\delta}[Z_\delta, G] \big\rangle \cong C_{p^{\gamma-\delta}}$ is a cyclic factor of $Z_\delta/[Z_\delta, G]$. Further analysis proceeds in two cases.
			
\noindent \textbf{Case 1 ($1 \leq \delta \leq \rho$).}				
			In this case, observe that $\big\langle a^{p^\delta}[Z_\delta, G] \big\rangle \cong C_{p^{\gamma-\delta}}$ and $\big\langle [a, b][Z_\delta, G] \big\rangle\cong C_{p^{\delta}}$ are also cyclic factors of $Z_\delta/[Z_\delta, G]$, yielding 
			$$Z_\delta/[Z_\delta, G] \cong C_{p^{\gamma-\delta}} \times C_{p^{\gamma-\delta}} \times C_{p^\delta}.$$

\noindent \textbf{Case 2 ($\rho < \delta \leq \gamma$).}
In this case, $${(a^{p^\delta})}^{p^{\gamma-\rho}}={(a^{p^\gamma})}^{p^{\delta-\rho}}={({[a, b]}^{p^\rho})}^{p^{\delta-\rho}}={[a, b]}^{p^\delta}.$$ 
\noindent \textbf{Sub-case 2(a) ($\rho < \delta<\gamma-\rho$).} In this sub-case, we have the following observations.
\begin{itemize}
\item $\order(a^{p^\delta}[Z_\delta, G])=p^{\gamma-\rho}>\order([a, b][Z_\delta, G])=p^\delta$ in $Z_\delta/[Z_\delta, G]$.
\item $\big\langle a^{p^\delta}[Z_\delta, G] \big\rangle \cong C_{p^{\gamma-\rho}}$ is also a cyclic factor of $Z_\delta/[Z_\delta, G]$.
\item ${(a^{p^\delta})}^{p^{\gamma-\delta}}=a^{p^\gamma}={[a, b]}^{p^\rho}$.
\end{itemize}
Therefore, in this sub-case, $$Z_\delta/[Z_\delta, G] \cong C_{p^{\gamma-\rho}} \times C_{p^{\gamma-\delta}} \times C_{p^\rho}.$$

\noindent \textbf{Sub-case 2(b) ($\gamma-\rho \leq \delta \leq \gamma$).} In this sub-case, we have the following observations.
\begin{itemize}
\item $\order(a^{p^\delta}[Z_\delta, G])=p^{\gamma-\rho} \leq \order([a, b][Z_\delta, G])=p^\delta$ in $Z_\delta/[Z_\delta, G]$.
\item $\big\langle [a, b][Z_\delta, G] \big\rangle \cong C_{p^\delta}$ is a cyclic factor of $Z_\delta/[Z_\delta, G]$.
\item ${(a^{p^\delta})}^{p^{\gamma-\delta}}=a^{p^\gamma}={[a, b]}^{p^\rho}$.
\end{itemize}
			Therefore, in this sub-case, we have
			$$Z_\delta/[Z_\delta, G] \cong C_{p^{\gamma-\delta}} \times C_{p^{\gamma-\delta}} \times C_{p^\delta}.$$
			
			\item The proof proceeds in the same way by replacing $[Z_\delta, G]$ with $[Z_{\delta-1}, G]$. Note that  $\big\langle b^{p^\delta}[Z_{\delta-1}, G] \big\rangle \cong C_{p^{\gamma-\delta}}$ is a cyclic factor of $Z_\delta/[Z_{\delta-1}, G]$. Further analysis again proceeds in two cases.
			
	\noindent\textbf{ Case 1 ($1 \leq \delta \leq \rho$).} In this case, observe that $\big\langle a^{p^\delta}[Z_{\delta-1}, G] \big\rangle \cong C_{p^{\gamma-\delta}}$ and $\big\langle [a, b][Z_{\delta-1}, G] \big\rangle \cong C_{p^{\delta-1}}$ are also cyclic factors of $Z_\delta/[Z_{\delta-1}, G]$. Therefore, we have 
			$$Z_\delta/[Z_{\delta-1}, G] \cong C_{p^{\gamma-\delta}} \times C_{p^{\gamma-\delta}} \times C_{p^{\delta-1}}.$$
			
		\noindent \textbf{Case 2 ($\rho < \delta \leq \gamma$).} In this case, we have
		    $${(a^{p^\delta})}^{p^{\gamma-\rho-1}}={(a^{p^\gamma})}^{p^{\delta-\rho-1}}={({[a, b]}^{p^\rho})}^{p^{\delta-\rho-1}}={[a, b]}^{p^{\delta-1}}.$$
		    
		    \noindent \textbf{Sub-case 2(a) ($\rho < \delta<\gamma-\rho$).} In this sub-case, we have the following observations.
		    \begin{itemize}
		    	\item $\order(a^{p^\delta}[Z_{\delta-1}, G])=p^{\gamma-\rho-1}>\order([a, b][Z_{\delta-1}, G])=p^{\delta-1}$ in $Z_\delta/[Z_{\delta-1}, G]$.
		    	
		    	\item $\big\langle a^{p^\delta}[Z_{\delta-1}, G] \big\rangle \cong C_{p^{\gamma-\rho}}$ is a cyclic factor of $Z_\delta/[Z_{\delta-1}, G]$.
		    	
		    	\item ${(a^{p^\delta})}^{p^{\gamma-\delta}}=a^{p^\gamma}={[a, b]}^{p^\rho}$.
		    \end{itemize}
			Hence, in this sub-case, we get
			$$Z_\delta/[Z_{\delta-1}, G] \cong C_{p^{\gamma-\rho-1}} \times C_{p^{\gamma-\delta}} \times C_{p^\rho}.$$
			
			\noindent \textbf{Sub-case 2(b) ($\gamma-\rho \leq \delta \leq \gamma$).} In this sub-case, we have the following observations.
			\begin{itemize}
				\item $\order(a^{p^\delta}[Z_{\delta-1}, G])=p^{\gamma-\rho-1} \leq \order([a, b][Z_{\delta-1}, G])=p^{\delta-1}$ in $Z_\delta/[Z_{\delta-1}, G]$.
				
				\item $\big\langle [a, b][Z_{\delta-1}, G] \big\rangle \cong C_{p^\delta}$ is a cyclic factor of $Z_\delta/[Z_{\delta-1}, G]$.
				
				\item ${(a^{p^\delta})}^{p^{\gamma-\delta}}=a^{p^\gamma}={[a, b]}^{p^\rho}$.
			\end{itemize}
			 Hence, in this sub-case, we get
			$$Z_\delta/[Z_{\delta-1}, G] \cong C_{p^{\gamma-\delta}} \times C_{p^{\gamma-\delta}} \times C_{p^{\delta-1}}.$$
		\end{enumerate}
		This completes the proof of Lemma~\ref{lemma:2generator5}.
	\end{proof}
	
	We are now in a position to prove Theorem~\ref{thm:Wedderburn2generator5}, which establishes an explicit combinatorial formula for the Wedderburn decomposition of rational group algebras associated with two-generator $p$-groups  $G_{(\alpha, \beta, \gamma; \rho, \sigma)}$ satisfying $(\alpha, \beta, \gamma; \rho, \sigma) \in \tau_{n_5}$.
	\begin{theorem}\label{thm:Wedderburn2generator5}
		Let $G= G_{(\alpha, \beta, \gamma; \rho, \sigma)}$ as defined in~\eqref{present:2generator p-group} be a two-generator $p$-group of class two and of order $p^n$ for an odd prime $p$, such that $(\alpha, \beta, \gamma; \rho, \sigma) \in \tau_{n_5}$. Then we have the following.
		\begin{enumerate}
			\item {\bf Case ($\gamma \leq 2 \rho+1$).} In this case, the Wedderburn decomposition of $\mathbb{Q}G$ is given by
			\begin{align*}
				\mathbb{Q}G \cong &\mathbb{Q} \bigoplus_{m=1}^{\gamma}(p^{m}+p^{m-1}) \mathbb{Q}(\zeta_{p^m}) \bigoplus_{\delta=1}^{\lfloor \frac{\gamma}{2}\rfloor} (p^{2\delta}+p^{2\delta-1}) M_{p^\delta}(\mathbb{Q}(\zeta_{p^\delta}))\\ &\bigoplus_{\delta=1}^{\lfloor \frac{\gamma}{2}\rfloor}\bigoplus_{m=\delta+1}^{\gamma-\delta} (p^{m+\delta}-p^{m+\delta-2})M_{p^\delta}(\mathbb{Q}(\zeta_{p^m})) \bigoplus_{\delta=\lfloor \frac{\gamma}{2}\rfloor +1}^\gamma p^{2(\gamma-\delta)} M_{p^\delta}(\mathbb{Q}(\zeta_{p^\delta})).
			\end{align*}
			
			\item {\bf Case ($\gamma > 2 \rho+1$).} In this case, the Wedderburn decomposition of $\mathbb{Q}G$ is given by
			\begin{align*}
				\mathbb{Q}G \cong &\mathbb{Q} \bigoplus_{m=1}^{\gamma}(p^{m}+p^{m-1}) \mathbb{Q}(\zeta_{p^m}) \bigoplus_{\delta=1}^{\rho} (p^{2\delta}+p^{2\delta-1}) M_{p^\delta}(\mathbb{Q}(\zeta_{p^\delta}))\\ &\bigoplus_{\delta=1}^{\rho}\bigoplus_{m=\delta+1}^{\gamma-\delta} (p^{m+\delta}-p^{m+\delta-2})M_{p^\delta}(\mathbb{Q}(\zeta_{p^m})) \bigoplus_{\delta=\rho+1}^{\gamma-\rho-1} p^{\gamma+\rho-\delta} M_{p^\delta}(\mathbb{Q}(\zeta_{p^{\gamma-\rho}}))\\
				& \bigoplus_{\delta=\gamma-\rho}^\gamma p^{2(\gamma-\delta)} M_{p^\delta}(\mathbb{Q}(\zeta_{p^\delta})).
			\end{align*}
		\end{enumerate}
	\end{theorem}
	
	\begin{proof} 
		Since $(\alpha, \beta, \gamma; \rho, \sigma) \in \tau_{n_5}$, we have $\alpha=\beta=\sigma=\gamma$ and $0\leq\rho \leq \gamma$. Thus, from \eqref{present:2generator p-group}, we have
		\begin{equation*}
			G=G_{(\gamma, \gamma, \gamma; \rho, \gamma)}=\langle a, b~:~[a, b]^{p^\gamma}=[a, b, a]=[a, b, b]=1, ~a^{p^\gamma}=[a, b]^{p^\rho}, ~b^{p^\gamma}=1 \rangle.
		\end{equation*}
		
		For $1 \leq \delta \leq \gamma$, we have 
		$$Z_\delta =\big \langle a^{p^\delta},~ b^{p^\delta},~ [a, b]\big\rangle; \quad [Z_\delta, G]=\big\langle {[a, b]}^{p^\delta} \big\rangle \quad \text{and} \quad [Z_{\delta-1}, G]=\big\langle {[a, b]}^{p^{\delta-1}} \big\rangle,$$
		where $Z_{\delta}=Z(\chi)$ for some $\chi\in \Irr_{p^{\delta}}(G)$ and $Z_{\delta-1}=Z(\psi)$ for some $\psi\in \Irr_{p^{\delta-1}}(G)$ (see Lemma \ref{lemma:property2generator}).
		
	    Furthermore, for $1 \leq \delta \leq \gamma$, let $a_{p^m}$ and $a_{p^m}'$ denote the number of cyclic subgroups of  order $p^m$ of $Z_\delta/[Z_\delta, G]$ and $Z_\delta/[Z_{\delta-1}, G]$, respectively. In view of Lemma~\ref{lemma:2generator5}, such groups are divided into two categories according to the values of $\gamma$ and $\rho$, and the proof is therefore considered in the following two cases.

		\begin{enumerate}
			\item {\bf Case ($\gamma \leq 2 \rho+1$).} In this case, observe that there does not exist any $\delta$ such that $\rho < \delta < \gamma-\rho$. Therefore, from Lemma~\ref{lemma:2generator5}), for each $\delta \in \{1,2,\dots,\gamma\}$, we have 
			\[
			Z_\delta/[Z_\delta, G] \cong C_{p^{\gamma-\delta}} \times C_{p^{\gamma-\delta}} \times C_{p^{\delta}} \quad \text{and} \quad Z_\delta/[Z_{\delta-1}, G] \cong C_{p^{\gamma-\delta}} \times C_{p^{\gamma-\delta}} \times C_{p^{\delta-1}}.
			\]
	We complete the rest of the proof in the following two sub-cases. 		
			
		\noindent \textbf{Sub-case ($\gamma-\delta \geq \delta$).}	 
			Note that 
			\[
			\gamma-\delta \geq \delta \iff \delta \leq \Big\lfloor \tfrac{\gamma}{2}\Big\rfloor,
			\]
			where $\lfloor \cdot \rfloor$ denotes the greatest integer function.\\  

			In this sub-case, we have $1 \leq \delta \leq \lfloor \tfrac{\gamma}{2}\rfloor$. By Lemma~\ref{lemma:CyclicCounting}, one can check that $a_{p^m} = a'_{p^m}$ for $1 \leq m \leq \delta-1$. Again from Lemma~\ref{lemma:CyclicCounting}, we have $k=3$ and
			\begin{align*}
				a_{p^\delta} = \frac{p h_p^2(\delta)- h_p^2(\delta-1)}{p-1},
			\end{align*}
			where $h^2_p(\delta)=p^{2\delta}$ and $h^2_p(\delta-1)=p^{2(\delta-1)}$. Hence, 
			\begin{align*}
				a_{p^\delta} = \frac{p^{2\delta+1}-p^{2\delta-2}}{p-1}.
			\end{align*}
			Similarly, 
			\begin{align*}
				a'_{p^\delta} &= \frac{p h_p^2(\delta)- h_p^2(\delta-1)}{p-1} 
				= \frac{p^{2\delta-1}-p^{2\delta-2}}{p-1}.
			\end{align*}
			as $h^2_p(\delta)=p^{2\delta-1}$ and $h^2_p(\delta-1)=p^{2(\delta-1)}$ (see Lemma~\ref{lemma:CyclicCounting}). Therefore, 
			\[
			a_{p^\delta}-a'_{p^\delta}=\frac{p^{2\delta+1}-p^{2\delta-1}}{p-1}=p^{2\delta}+p^{2\delta-1}.
			\]  
			Now, for $\delta+1 \leq m \leq \gamma-\delta$, Lemma~\ref{lemma:CyclicCounting} gives
			\begin{align*}
				a_{p^m} = \frac{p^{m+\delta+1}-p^{m+\delta-1}}{p-1} \quad \text{and} \quad
		       a'_{p^m} = \frac{p^{m+\delta}-p^{m+\delta-2}}{p-1}.
			\end{align*}
			Hence,
			\[
			a_{p^m}-a'_{p^m}= p^{m+\delta}-p^{m+\delta-2}.
			\]   

\noindent \textbf{Sub-case ($\gamma-\delta < \delta$).} Note that
\[
\gamma-\delta < \delta \iff \delta > \Big\lfloor \tfrac{\gamma}{2}\Big\rfloor.
\]
For each $\lfloor \tfrac{\gamma}{2}\rfloor +1 \leq \delta \leq \gamma$, Lemma~\ref{lemma:CyclicCounting} shows that $a_{p^m}=a'_{p^m}$ for all $1 \leq m \leq \delta-1$. Furthermore, $k=3$ and 
\[
a_{p^\delta} = \frac{p h_p^2(\delta)- h_p^2(\delta-1)}{p-1} 
= \frac{p^{2(\gamma-\delta)+1}-p^{2(\gamma-\delta)}}{p-1}
= p^{2(\gamma-\delta)},
\]  
as $h^2_p(\delta)= h^2_p(\delta-1)=p^{2(\gamma-\delta)}$ (see Lemma~\ref{lemma:CyclicCounting}). Also, observe that $a'_{p^{\delta}}=0$.
Thus, the result follows from Theorem~\ref{thm:Wedderburn2generator}.

			\item {\bf Case ($\gamma > 2 \rho+1$).} In this case, there exists some $\delta$ such that $\rho < \delta < \gamma-\rho$. We proceed the proof in the following three sub-cases.		
				
\noindent \textbf{Sub-case ($1 \leq \delta \leq \rho$).} In this sub-case, from Lemma~\ref{lemma:2generator5}, we have
			\[
			Z_\delta/[Z_\delta, G] \cong C_{p^{\gamma-\delta}} \times C_{p^{\gamma-\delta}} \times C_{p^{\delta}}
			\quad \text{and} \quad
			Z_\delta/[Z_{\delta-1}, G] \cong C_{p^{\gamma-\delta}} \times C_{p^{\gamma-\delta}} \times C_{p^{\delta-1}}.
			\]
			Note that in this sub-case, $\gamma-\delta \geq \delta$. Hence, for each $1 \leq \delta \leq \rho$, by Lemma~\ref{lemma:CyclicCounting}, we have $a_{p^m}=a'_{p^m}$ whenever $1 \leq m \leq \delta-1$. Moreover, analogous to the previous case, we have
			\begin{align*}
				a_{p^\delta} = \frac{p^{2\delta+1}-p^{2\delta-2}}{p-1} \quad \text{and} \quad
				a'_{p^\delta} = \frac{p^{2\delta-1}-p^{2\delta-2}}{p-1}.
			\end{align*}
			Therefore, we get
			\[
			a_{p^\delta}-a'_{p^\delta}=\frac{p^{2\delta+1}-p^{2\delta-1}}{p-1}=p^{2\delta}+p^{2\delta-1}.
			\]  
	Further, for $\delta+1 \leq m \leq \gamma-\delta$, analogous to the previous case, Lemma~\ref{lemma:CyclicCounting} yields
			\begin{align*}
				a_{p^m}= \frac{p^{m+\delta+1}-p^{m+\delta-1}}{p-1} \quad \text{and} \quad
				a'_{p^m}= \frac{p^{m+\delta}-p^{m+\delta-2}}{p-1}.
			\end{align*}
			Hence, 
			\[
			a_{p^m}-a'_{p^m}= p^{m+\delta}-p^{m+\delta-2}.
			\]  
			
\noindent \textbf{Sub-case ($\rho+1 \leq \delta \leq \gamma-\rho-1$).}
In this sub-case, from Lemma~\ref{lemma:2generator5}, we have
			\[
			Z_\delta/[Z_\delta, G] \cong C_{p^{\gamma-\rho}} \times C_{p^{\gamma-\delta}} \times C_{p^{\rho}}
			\quad \text{and} \quad
			Z_\delta/[Z_{\delta-1}, G] \cong C_{p^{\gamma-\rho-1}} \times C_{p^{\gamma-\delta}} \times C_{p^{\rho}}.
			\]
			  Since $\gamma-\rho > \rho+1$ and $\delta\in \{\rho+1, \ldots, \gamma-\rho-1\}$, one can see that $\rho+1\leq \gamma-\delta \leq \gamma-\rho-1$.			  
Hence, in this sub-case, $\rho < \gamma-\delta < \gamma-\rho$. By Lemma~\ref{lemma:CyclicCounting}, we have $a_{p^m}=a'_{p^m}$ for $1 \leq m \leq \gamma-\rho-1$. Again from Lemma~\ref{lemma:CyclicCounting}, we have $k=3$ and
			\[
			a_{p^{\gamma-\rho}} = \frac{p h_p^2(\gamma-\rho)- h_p^2(\gamma-\rho-1)}{p-1} 
			= \frac{p^{\gamma-\delta+\rho+1}-p^{\gamma-\delta+\rho}}{p-1}
			= p^{\gamma+\rho-\delta},
			\]  
			as $h^2_p(\gamma-\rho)= h^2_p(\gamma-\rho-1)=p^{\gamma-\delta+\rho}$. Further, note that $a'_{p^{\gamma-\rho}}=0$. 
			
\noindent \textbf{Sub-case ($\gamma-\rho \leq \delta \leq \gamma$).} In this sub-case, we have
			\[
			Z_\delta/[Z_\delta, G] \cong C_{p^{\gamma-\delta}} \times C_{p^{\gamma-\delta}} \times C_{p^{\delta}} \quad \text{and} \quad
			Z_\delta/[Z_{\delta-1}, G] \cong C_{p^{\gamma-\delta}} \times C_{p^{\gamma-\delta}} \times C_{p^{\delta-1}}
			\]
			(see Lemma~\ref{lemma:2generator5}). In this range of $\delta$, we have $\gamma-\delta \leq \delta$. Now,  by Lemma~\ref{lemma:CyclicCounting}, we again obtain $a_{p^m}=a'_{p^m}$ for $1 \leq m \leq \delta-1$. Furthermore, we have $k=3$ and
			\[
			a_{p^\delta} = \frac{p h_p^2(\delta)- h_p^2(\delta-1)}{p-1} 
			= \frac{p^{2(\gamma-\delta)+1}-p^{2(\gamma-\delta)}}{p-1}
			= p^{2(\gamma-\delta)},
			\]  
			as $h^2_p(\delta)= h^2_p(\delta-1)=p^{2(\gamma-\delta)}$ (see Lemma~\ref{lemma:CyclicCounting}). Again, note that $a'_{p^{\delta}}=0$. Thus, the result follows from Theorem~\ref{thm:Wedderburn2generator}.  
		\end{enumerate}
		This completes the proof of Theorem~\ref{thm:Wedderburn2generator5}.
	\end{proof}
	
	Next, we have Corollary~\ref{cor:2generators5}.
	\begin{corollary}\label{cor:2generators5}
		Let $G_1 = G_{(\gamma, \gamma, \gamma; \rho_1, \gamma)}$ and 
		$G_2 = G_{(\gamma, \gamma, \gamma; \rho_2, \gamma)}$ be $p$-groups of class two, each generated by two elements and of order $p^n$, for an odd prime $p$, as defined in~\eqref{present:2generator p-group}. Suppose that $\frac{\gamma-1}{2} \leq \min\{\rho_1,\, \rho_2\}$. Then $\mathbb{Q}G_1 \cong \mathbb{Q}G_2$.
	\end{corollary}
	\begin{proof}
		It follows directly from the first case of Theorem~\ref{thm:Wedderburn2generator5}.
	\end{proof}
	
	We conclude this section with the following example. This example shows that two non-isomorphic, two-generator $p$-groups of class $2$ can have isomorphic rational group algebras if they satisfy the hypothesis of Corollary~\ref{cor:2generators5}. However, this need not hold if the hypothesis of Corollary~\ref{cor:2generators5} is relaxed. 
	
	\begin{example}\label{example:2generator5}
		\textnormal{
			Let 
			\[
			G_1 = \langle a, b : [a, b]^{9} = [a, b, a] = [a, b, b] = 1,\; a^{9} = 1,\; b^{9} = 1 \rangle
			\]
			and
			\[
			G_2 = \langle a, b : [a, b]^{9} = [a, b, a] = [a, b, b] = 1,\; a^{9} = {[a, b]}^{3},\; b^{9} = 1 \rangle
			\]
			be two groups.  Note that both $G_1$ and $G_2$ are two-generator $3$-groups of nilpotency class $2$ and of order $729$, with $G_1 = G_{(2,2,2;2,2)}$ and $G_2 = G_{(2,2,2;1,2)}$
			(see \eqref{present:2generator p-group}). Hence, by the first case of Theorem~\ref{thm:Wedderburn2generator5}, we obtain
			$$\mathbb{Q}G_1 \cong \mathbb{Q}G_2  \cong \mathbb{Q} \oplus 4\mathbb{Q}(\zeta_3) \oplus 12 \mathbb{Q}(\zeta_9) \oplus 9 M_3(\mathbb{Q}(\zeta_3)) \oplus M_9(\mathbb{Q}(\zeta_9)). $$
			Next, let 
			\[
			G_3 = \langle a, b : [a, b]^{9} = [a, b, a] = [a, b, b] = 1,\; a^{9} = [a, b],\; b^{9} = 1 \rangle
			\]
			be a group. Note that $G_3$ is a two-generator $3$-groups of nilpotency class $2$ and of order $729$, with $G_3 = G_{(2,2,2;0,2)}$ (see \eqref{present:2generator p-group}). Hence, by the second case of Theorem~\ref{thm:Wedderburn2generator5}, we obtain
			$$\mathbb{Q}G_3 \cong \mathbb{Q} \oplus 4\mathbb{Q}(\zeta_3) \oplus 12 \mathbb{Q}(\zeta_9) \oplus 3 M_3(\mathbb{Q}(\zeta_9)) \oplus M_9(\mathbb{Q}(\zeta_9)). $$
			Furthermore, $G_1$, $G_2$ and $G_3$ correspond to $\mathrm{SmallGroup}(729,24)$, $\mathrm{SmallGroup}(729,25)$ and $\mathrm{SmallGroup}(729,60)$, respectively, in the {\sf GAP} SmallGroups library. The above Wedderburn decomposition can also be verified using the \textsc{Wedderga} package in {\sf GAP}.
		}
	\end{example}

	\section{Nested GVZ $p$-groups of arbitrarily large nilpotency class}\label{sec:GVZarbitrary}
	
	It is known that GVZ $p$-groups exist with arbitrarily large nilpotency class. In this section, we compute the Wedderburn decomposition of rational group algebras corresponding to two distinct families of nested GVZ $p$-groups with arbitrarily large nilpotency class. In Subsection~\ref{subsec:GVZ1}, we determine the Wedderburn decomposition of rational group algebras of a family of nested GVZ $p$-groups of arbitrarily large nilpotency class and exponent $p$. In Subsection~\ref{subsec:GVZ2}, we deal with the Wedderburn decomposition of rational group algebras for a family of nested GVZ $p$-groups of arbitrarily large nilpotency class and arbitrarily large exponent.
	
	\subsection{Nested GVZ $p$-groups of arbitrary large nilpotency class and exponent $p$}\label{subsec:GVZ1}
	We begin with a family of nested GVZ $p$-groups introduced by Nenciu~\cite{NenciuGVZ2}. The members of this family are $p$-groups of exponent $p$. For every $n \geq 1$, we construct a group $G_n$, which is a nested GVZ-group of order $p^{2n+1}$, exponent $p$, and nilpotency class $n+1$, where $p > n+1$ is a prime. 
	
	Let $G_1$ be the extra-special $p$-group of order $p^3$ and exponent $p$. We may describe $G_1$ as
	$$G_1=(\langle b_0 \rangle \times \langle a_1 \rangle) \rtimes \langle b_1 \rangle,$$
	with the relation $[b_1, a_1]=b_0$. Assume that $G_{n-1}$ is defined. Now, we construct $G_n$ by setting
	$$G_n=(G_{n-1} \times \langle a_n \rangle) \rtimes \langle b_n \rangle,$$
	where $a_n^p=b_n^p,~ [b_n, a_1]=b_{n-1},~ [a_2, b_n]=b_{n-2}, \ldots, [a_{n-1}, b_n]=b_1,~ [a_n, b_n]=b_0$, and all remaining commutators are trivial. 
	
	In terms of generators and relations, we obtain
	\begin{equation}\label{presentation:GVZ1}
		G_n=\langle a_1, a_2, \ldots, a_n, b_0, b_1, \ldots, b_n \rangle,
	\end{equation}
	where each generator has order $p$, $[a_i, a_j]=[b_i, b_j]=1$ for all $0\leq i<j\leq n$, and
	\begin{equation*}
		[a_i, b_j] = \begin{cases}
			1 & \text{if } i > j, \\
			b_{j-1} & \text{if } 1 < i \leq j,\\
			b_{j-1}^{p-1} & \text{if } 1=i \leq j.
		\end{cases}
	\end{equation*}
	
	We now recall the following result from~\cite{NenciuGVZ2}.
	\begin{lemma}\cite[Theorem~5.1]{NenciuGVZ2}\label{lemma:propertyGVZ1}
		For every $n \geq 1$, the group $G_n$ defined in~\eqref{presentation:GVZ1} is a nested GVZ-group with the following properties.
		\begin{enumerate}
			\item $\cd(G_n)=\{1, p, \ldots, p^n\}$.
			\item If $\chi \in \Irr(G_n)$ with $\chi(1)=p^r$ for some $r \in \{0,1,\ldots,n\}$, then
			\begin{enumerate}
				\item $Z(\chi)=\langle a_{r+1}, \ldots, a_n, b_0, b_1, \ldots, b_{n-r}\rangle$,
				\item $\langle b_0, b_1, \ldots, b_{n-r-1}\rangle \subseteq \ker(\chi)$.
			\end{enumerate}
		\end{enumerate}
	\end{lemma}
	
	We are now ready to prove the following theorem, which gives a combinatorial formula for the Wedderburn decomposition of rational group algebras of the family of nested GVZ $p$-groups described in~\eqref{presentation:GVZ1}.
	\begin{theorem}\label{thm:WedderburnGVZ1}
		Let $G=G_n$ be as defined in~\eqref{presentation:GVZ1}, a nested GVZ $p$-group of order $p^{2n+1}$, exponent $p$, and class $n+1$, where $p> n+1$ is an odd prime. Then the Wedderburn decomposition of $\mathbb{Q}G$ is given by
		$$\mathbb{Q}G \cong \mathbb{Q} \bigoplus (1+p+p^2+\cdots+p^n)\,\mathbb{Q}(\zeta_p) \bigoplus_{r=1}^n p^{\,n-r}\, M_{p^r}(\mathbb{Q}(\zeta_p)).$$
	\end{theorem}
	\begin{proof}
		Let $G=G_n$ be defined as above. Note that $G'=\langle b_0, b_1, \ldots, b_{n-1}\rangle$. Hence,
		$$G/G'=\langle a_1G', a_2G', \ldots, a_nG', b_nG'\rangle \cong \underbrace{C_p \times \cdots \times C_p}_{n+1\ \text{factors}}.$$ 
		By Theorem~\ref{Perlis-walker} and Lemma~\ref{lemma:CyclicCounting}, it follows that
		$$\mathbb{Q}(G/G') \cong \mathbb{Q} \bigoplus (1+p+p^2+\cdots+p^n)\,\mathbb{Q}(\zeta_p).$$
		
		Moreover, $\cd(G)=\{1,p,\ldots,p^n\}$. For $\chi \in \nl(G)$ with $\chi(1)=p^r$ ($1 \leq r \leq n$), set $Z_{r}=Z(\chi)$. From Lemma~\ref{lemma:propertyGVZ1}, we obtain
		$$Z_{r}=\langle a_{r+1}, \ldots, a_n, b_0, b_1, \ldots, b_{n-r}\rangle,$$
		and
		$$[Z_{r}, G]=\langle b_0, b_1, \ldots, b_{n-r-1}\rangle.$$
		Hence, we have the quotient group
		$$Z_{r}/[Z_{r}, G] = \big \langle a_{r+1}[Z_{r}, G], \ldots, a_n[Z_{r}, G],~ b_{n-r}[Z_{r}, G] \big \rangle \cong \underbrace{C_p \times \cdots \times C_p}_{n-r+1\ \text{factors}}.$$
		Also, since $[Z_{r-1}, G]=\langle b_0, b_1, \ldots, b_{n-r}\rangle$, we obtain the quotient group
		$$Z_{r}/[Z_{r-1}, G] \cong \underbrace{C_p \times \cdots \times C_p}_{n-r\ \text{factors}}.$$
		
		Thus, by applying Theorem~\ref{thm:WedderburnGVZ} together with Lemma~\ref{lemma:CyclicCounting}, the claimed decomposition follows. This completes the proof of Theorem~\ref{thm:WedderburnGVZ1}.
	\end{proof}

	\subsection{Nested GVZ $p$-groups of arbitrary large nilpotency class and arbitrary large exponent}\label{subsec:GVZ2}
	
	In the previous subsection, we considered a family of nested GVZ $p$-groups of exponent $p$. We now turn to another important family, introduced by Lewis~\cite{LewisGVZ2}, which provides nested GVZ $p$-groups of arbitrarily large nilpotency class and simultaneously arbitrarily large exponent. Let $p$ be an odd prime and $n$ a positive integer. Denote by $C_{p^{n+1}}$ the cyclic group of order $p^{n+1}$. It is well known that the Sylow $p$-subgroup of $\aut(C_{p^{n+1}})$ is cyclic of order $p^n$. Define
	$$G_n=C_{p^{n+1}} \rtimes \Syl_p(\aut(C_{p^{n+1}})).$$
	Then $|G_n|=p^{2n+1}$. Let $C_{p^{n+1}}=\langle x \rangle$ and $\Syl_p(\aut(C_{p^{n+1}}))=\langle y \rangle$. With these generators, a presentation of $G_n$ is
	\begin{equation}\label{presentation:GVZ2}
		G_n=\langle x, y : x^{p^{n+1}}=y^{p^n}=1,~ y^{-1}xy=x^{1+p}\rangle.
	\end{equation}
	Here, $G_n'=\langle x^p \rangle \cong C_{p^n}$ and $Z(G_n)=\langle x^{p^n} \rangle \cong C_p$. Moreover, $\cd(G)=\{1, p, p^2, \ldots, p^n\}$. Next, we have Lemma~\ref{lemma:propertyGVZ2}, which summaries some properties of $G_n$.
	
	\begin{lemma}\cite[Example~3]{LewisGVZ2}\label{lemma:propertyGVZ2}
		For every $n \geq 1$, the group $G=G_n$ defined in~\eqref{presentation:GVZ2} is a nested GVZ $p$-group of order $p^{2n+1}$ with the following properties:
		\begin{enumerate}
			\item $G$ has nilpotency class $n+1$ and exponent $p^{n+1}$.
			\item $\cd(G)=\{1, p, \ldots, p^n\}$.
			\item If $\chi \in \Irr(G)$ with $\chi(1)=p^r$ for some $r \in \{0,1,\ldots,n\}$, then 
			$$Z(\chi)=\langle x^{p^r}, y^{p^r}\rangle \quad \text{and} \quad [Z(\chi), G]=\langle x^{p^{r+1}} \rangle.$$
		\end{enumerate}
	\end{lemma}
	
	We now establish the Wedderburn decomposition of rational group algebras for this family.
	
	\begin{theorem}\label{thm:WedderburnGVZ2}
		Let $G=G_n=\langle x, y : x^{p^{n+1}}=y^{p^n}=1,~ y^{-1}xy=x^{1+p}\rangle$ be a nested GVZ $p$-group of order $p^{2n+1}$, exponent $p^{n+1}$, and nilpotency class $n+1$, where $p$ is an odd prime. Then the Wedderburn decomposition of $\mathbb{Q}G$ is given by
		\begin{align*}
			\mathbb{Q}G \cong \mathbb{Q} \bigoplus (p+1)\mathbb{Q}(\zeta_p) \bigoplus_{r=2}^{n}p \mathbb{Q}(\zeta_{p^r}) \bigoplus_{r=1}^{n-1} pM_{p^r}(\mathbb{Q}(\zeta_{p})) \bigoplus_{r=1}^{n-2}\bigoplus_{m=2}^{n-r} (p-1)M_{p^r}(\mathbb{Q}(\zeta_{p^m})) \bigoplus M_{p^n}(\mathbb{Q}(\zeta_{p})).
		\end{align*}
	\end{theorem}
	
	\begin{proof}
		Consider $G=G_n$ as defined in~\eqref{presentation:GVZ2}. Then $G'=\langle x^p \rangle \cong C_{p^n}$, and hence
		$$G/G'=\langle xG', yG'\rangle \cong C_p \times C_{p^n}.$$
		Therefore, by Theorem~\ref{Perlis-walker} and Lemma~\ref{lemma:CyclicCounting}, we obtain
		$$\mathbb{Q}(G/G') \cong \mathbb{Q} \bigoplus (p+1)\mathbb{Q}(\zeta_p) \bigoplus_{r=2}^n p\,\mathbb{Q}(\zeta_{p^r}).$$
		
		Next, note that $\cd(G)=\{1,p,\ldots,p^n\}$. For $\chi \in \nl(G)$ with $\chi(1)=p^r$ ($1 \leq r \leq n$), define $Z_{r}=Z(\chi)$. Then by Lemma~\ref{lemma:propertyGVZ2}, we have 
		$$Z_{r}=\langle x^{p^r}, y^{p^r}\rangle \quad \text{and} \quad [Z_{p^r}, G]=\langle x^{p^{r+1}}\rangle.$$
		Thus, for $1 \leq r \leq n$, we have the quotient group
		$$Z_{r}/[Z_{r}, G] = \big \langle x^{p^r}[Z_{r}, G], ~y^{p^r}[Z_{r}, G] \big \rangle \cong C_p \times C_{p^{\,n-r}}.$$
		Similarly, we have $[Z_{r-1}, G]=\langle x^{p^{r}}\rangle$ (see Lemma~\ref{lemma:propertyGVZ2}). Thus, the quotient group
		$$Z_{r}/[Z_{r-1}, G] =\big \langle y^{p^r}[Z_{r}, G] \big \rangle \cong C_{p^{\,n-r}}.$$

Let $a_{p^s}$ and $a'_{p^s}$ denote the number of cyclic subgroups of order $p^s$ in $Z_{r}/[Z_{r}, G]$ and $Z_{r}/[Z_{r-1}, G]$, respectively. 
For $1\leq r\leq n-1$ and $1\leq s\leq n-r$, by Lemma~\ref{lemma:CyclicCounting}, we have 
$$a_{p^s}=\begin{cases}p+1 & \text{if}~ s=1,\\
                p & \text{if}~ 2\leq s\leq n-r
\end{cases}
\quad \text{and} \quad a'_{p^s}=1. $$
Furthermore, for $r=n$, we have $a_p=1$.  Therefore, by applying Theorem~\ref{thm:WedderburnGVZ}, we obtain the claimed Wedderburn decomposition of $\mathbb{Q}G$. This completes the proof of Theorem~\ref{thm:WedderburnGVZ2}.
	\end{proof}
	
	\begin{remark}
		\textnormal{The family of groups described in~\eqref{presentation:GVZ2} are split metacyclic $p$-groups. The decomposition formula for the Wedderburn decomposition of rational group algebras of this family of groups given in Theorem~\ref{thm:WedderburnGVZ2} can alternatively be deduced from~\cite[Theorem~1]{Ram3}.}
	\end{remark}

\section{Isoclinic nested GVZ $p$-groups}\label{sec:isoGVZ}
In this section, we first establish that the property of being a GVZ-group (respectively, a nested GVZ-group) is preserved under \emph{isoclinism}. Furthermore, we classify all nested GVZ $p$-groups of order at most $p^5$. We then explicitly compute the Wedderburn decomposition of the rational group algebras of all nested GVZ $p$-groups of order $\leq p^5$, where $p$ is an odd prime. These computations, carried out using our main theorem, provide one more concrete illustration of Theorem~\ref{thm:WedderburnGVZ}. We begin with the following definition.

\begin{definition}
	Two finite groups $G$ and $H$ are said to be \emph{isoclinic} if there exist two isomorphisms $\theta : G/Z(G) \longrightarrow H/Z(H)$ and $\phi : G' \longrightarrow H'$
	such that for all $g_1, g_2 \in G$, if $h_1 \in \theta(g_1Z(G))$ and $h_2 \in \theta(g_2Z(G))$, then
	\[
	\phi([g_1, g_2]) = [h_1, h_2].
	\]
\end{definition}

The pair $(\theta, \phi)$ is referred to as an \emph{isoclinism} from $G$ onto $H$. The concept was originally introduced by Hall~\cite{PH} in the context of classifying $p$-groups and may be viewed as a generalization of group isomorphism. It is a standard fact that two isoclinic nilpotent groups share the same nilpotency class.

We now proceed to the proof of Theorem~\ref{thm:isoGVZ}.
\begin{proof}[Proof of Theorem~\ref{thm:isoGVZ}]
	Let $G$ and $H$ be two finite isoclinic groups. Suppose that $G$ is a GVZ-group. Since $G$ and $H$ are isoclinic, there exist two isomorphisms $\theta : G/Z(G) \longrightarrow H/Z(H)$ and $\phi : G' \longrightarrow H'$ such that 
	\[
	\phi([g_1, g_2]) = [h_1, h_2]
	\]
	whenever $\theta(g_1Z(G)) = h_1Z(H)$ and $\theta(g_2Z(G)) = h_2Z(H)$ for $g_1, g_2 \in G$ and $h_1, h_2 \in H$.
	
	Let $h \in H$. Choose $g \in G$ such that $\theta(gZ(G)) = hZ(H)$. Since $\theta$ is surjective, we obtain
	\[
	\phi(\{[g, x] : x \in G\}) = \{\phi([g, x]) : x \in G\} = \{[h, y] : y \in H\}.
	\]
	
	Recall that $G$ is a GVZ-group if and only if $\cl_G(g) = g[g, G]$ for every $g \in G$ (see \cite{Burkett}). In particular,
	\[
	[g, G] = \{[g, x] : x \in G\}.
	\]
	Hence,
	\[
	\phi([g, G]) = \{[h, y] : y \in H\}.
	\]
	This shows that $\{[h, y] : y \in H\}$ is a subgroup of $H'$, i.e., $[h, H] = \{[h, y] : y \in H\}$.
	Consequently, $\cl_H(h) = h[h, H]$, and therefore $H$ is also a GVZ-group.
	
	Next, let $G$ is a nested GVZ-group. Let $h_1, h_2 \in H$, and choose $g_1, g_2 \in G$ such that $\theta(g_1Z(G)) = h_1Z(H)$ and $\theta(g_2Z(G)) = h_2Z(H)$. Since $G$ is a nested GVZ-group, it follows that either $[g_1, G] \leq [g_2, G]$ or $[g_2, G] \leq [g_1, G]$ (see \cite{Burkett2}). Without loss of generality, suppose $[g_1, G] \leq [g_2, G]$. Then we have
	\[
	[g_1, G] \leq [g_2, G] 
	\implies \phi([g_1, G]) \leq \phi([g_2, G]) 
	\implies [h_1, H] \leq [h_2, H].
	\]
	Thus, $H$ is also a nested GVZ-group (see~\cite[Theorem~2]{Burkett2}). This completes the proof of Theorem~\ref{thm:isoGVZ}.
\end{proof}

It is well known that a non-abelian $p$-group $G$ of order at most $p^4$ is a nested GVZ-group if and only if the nilpotency class of $G$ is $2$. The groups of order $p^5$ for an odd prime $p$ were classified by James~\cite{RJ} based on isoclinism, and we adopt the notations and presentations of these groups as given in~\cite{RJ}. This classification yields $10$ isoclinism families, denoted by $\Phi_i$ for $1 \leq i \leq 10$. We now establish the proof of Corollary~\ref{cor:isoGVZp^5}, which provides a classification of all nested GVZ $p$-groups of order $p^5$. Before proceeding, we recall a fundamental result concerning GVZ-groups.

\begin{lemma}\cite[Theorem~B]{Burkett}\label{lemma:nilclassGVZ}
	If $G$ is a GVZ-group, then $\nil(G) \leq |\cd(G)|$.
\end{lemma}

\begin{proof}[Proof of Corollary~\ref{cor:isoGVZp^5}]
	Let $G$ be a non-abelian group of order $p^5$, where $p$ is an odd prime. Then $|\cd(G)|=2$ except when $G \in \Phi_{7} \cup \Phi_{8} \cup \Phi_{10}$ (see \cite[Subsection~4.1]{RJ}). Note that a nested GVZ $p$-group $G$ with $|\cd(G)|=2$ is, in fact, a VZ $p$-group. It is well known that a non-abelian $p$-group $G$ of order $p^5$ is a VZ-group if and only if $G \in \Phi_{2} \cup \Phi_{5}$.
	
	Next, suppose $G \in \Phi_{10}$. In this case, we have $\cd(G)=\{1, p, p^2\}$ and the nilpotency class of $G$, denoted by $\nil(G)$, is $4$. Thus, $|\cd(G)| < \nil(G)$. Therefore, by Lemma~\ref{lemma:nilclassGVZ}, $G$ is not a GVZ-group.
	
	Finally, suppose $G \in \Phi_{7} \cup \Phi_{8}$. Then we have $\cd(G)=\{1, p, p^2\}$, $Z(G)\cong C_p$, and
	\[	G/Z(G)\cong 
	\begin{cases}
		\Phi_{2}(1^4) & \text{if } G \in \Phi_{7}, \\
		\Phi_{2}(22) & \text{if } G \in \Phi_{8}
	\end{cases}
	\]
	(see \cite[Subsection~4.1]{RJ}). Note that there is a bijection between the sets $\Irr_{p}(G)$ and $\nl(G/Z(G))$. Let $\chi \in \Irr_{p}(G)$ be the lift of $\bar{\chi} \in \nl(G/Z(G))$. Observe that $G/Z(G)$ is a VZ-group, and hence $\bar{\chi} \in \nl(G/Z(G))$ is of central type. This implies that $\chi \in \Irr_{p}(G)$ is of central type.
	
	Furthermore, the pair $(G, Z(G))$ is a Camina pair (see \cite[Lemmas~5.5 and~5.6]{SKP}). Therefore, there is a bijection between the sets $\Irr_{p^2}(G)$ and $\Irr(Z(G)) \setminus \{1_{Z(G)}\}$, where $1_{Z(G)}$ denotes the trivial character of $Z(G)$. Moreover, for $1_{Z(G)} \neq \mu \in \Irr(Z(G))$, the corresponding $\chi_\mu \in \Irr_{p^2}(G)$ is given by
	\[
	\chi_\mu(g) = 
	\begin{cases}
		p^2 \mu(g) & \text{if } g \in Z(G),\\
		0          & \text{otherwise.}
	\end{cases}
	\]
	Thus, $Z(\chi_\mu)=Z(G)$. This implies that 
	\[
	|G/Z(\chi_\mu)|^{\frac{1}{2}}=p^2=\chi_\mu(1).
	\]
	Therefore, $\chi_\mu \in \Irr_{p^2}(G)$ is of central type. Moreover, notice that $G$ is strictly nested by degrees. Hence, $G \in \Phi_{7} \cup \Phi_{8}$ is a nested GVZ-group. This completes the proof of Corollary~\ref{cor:isoGVZp^5}.
\end{proof}

Recall that a nilpotent group of class~$2$ is a GVZ-group (see~\cite[Theorem~2.31]{I}). 
The groups of order $p^5$ in the isoclinic family $\Phi_{4}$ have nilpotency class~$2$ (see~\cite[Subsection~4.1]{RJ}). Hence, these groups are GVZ-groups that are not nested GVZ-groups. We now prove Theorem~\ref{thm:WedderburnGVZp^5}, which gives the Wedderburn decomposition of rational group algebras for all nested GVZ $p$-groups of order at most $p^5$.
\begin{theorem}\label{thm:WedderburnGVZp^5}
	Let $G$ be a non-abelian nested GVZ $p$-group of order $\leq p^5$, where $p$ is an odd prime. Then $\nil(G) \in \{2,3\}$. Moreover, we have the following.
	\begin{enumerate}
		\item {\bf Case ($\nil(G)=2$).} In this case, let $p^m$ and $p^n$ denote the exponents of $G/G'$ and $Z(G)$, respectively. Then the Wedderburn decomposition of $\mathbb{Q}G$ is given by
		$$\mathbb{Q}G \cong \bigoplus_{\lambda=0}^m a_{p^\lambda}\mathbb{Q}(\zeta_{p^\lambda}) \bigoplus_{\lambda=1}^{n}(b_{p^\lambda}-c_{p^\lambda})M_{|G/Z(G)|^{1/2}}\!\left(\mathbb{Q}(\zeta_{p^\lambda})\right),$$
		where $a_{p^\lambda}$, $b_{p^\lambda}$ and $c_{p^\lambda}$ are the number of cyclic subgroups of order $p^\lambda$ of $G/G'$, $Z(G)$ and $Z(G)/G'$, respectively.
		
		\item {\bf Case ($\nil(G)=3$).} In this case, let $p^m$ denote the exponent of $G/G'$. Then the Wedderburn decomposition of $\mathbb{Q}G$ is given by
		\[\mathbb{Q}G \cong \bigoplus_{\lambda=0}^m a_{p^\lambda}\mathbb{Q}(\zeta_{p^\lambda}) \bigoplus p M_p(\mathbb{Q}(\zeta_p)) \bigoplus M_{p^2}(\mathbb{Q}(\zeta_p)),\]
		where $a_{p^\lambda}$ is the number of cyclic subgroups of order $p^\lambda$ of $G/G'$.
	\end{enumerate}
\end{theorem}
\begin{proof}
	Let $G$ be a non-abelian nested GVZ $p$-group of order at most $p^{5}$, where $p$ is an odd prime. Then $\cd(G) \subseteq \{1, p, p^{2}\}$. Hence, by Lemma~\ref{lemma:nilclassGVZ}, we obtain that $\nil(G) \in \{2,3\}$. 
	\begin{enumerate}
		\item First, observe that every non-abelian nested GVZ $p$-group of order at most $p^{5}$ with nilpotency class $2$ is a VZ-group. Therefore, by Theorems~\ref{Perlis-walker} and~\ref{lemma:Wedderburn VZ}, the result follows in this case.
		
		\item Now suppose that $G$ is a nested GVZ-group of nilpotency class $3$. Then necessarily $|G| = p^{5}$ and $G \in \Phi_{7} \cup \Phi_{8}$ (see the proof of Corollary~\ref{cor:isoGVZp^5}). For $i \in \{0,1,2\}$, let $Z_i = Z(\chi)$ for some $\chi \in \Irr_{p^{i}}(G)$. In addition, let $N$ be a normal subgroup of $G$ such that
		$$\{e\} < Z(G) < N < G$$
		is part of the upper central series, where $e$ denotes the identity element of $G$. Since $G/N = Z(G/N)$, the quotient $G/N$ is abelian, which implies $G' \subseteq N$. Moreover, we have $N/Z(G) = Z(G/Z(G))$. By \cite[Subsection~4.1]{RJ},
		$$
		G/Z(G) \cong 
		\begin{cases}
			\Phi_{2}(1^{4}) & \text{if } G \in \Phi_{7}, \\
			\Phi_{2}(22) & \text{if } G \in \Phi_{8}.
		\end{cases}
		$$
		There is a natural bijection between $\Irr_{p}(G)$ and $\nl(G/Z(G))$. Since $G/Z(G)$ is a VZ-group, it follows that $Z_{1} = N$. Next, $Z_{2} = Z(G)$ as $\cd(G) = \{1,p,p^{2}\}$. Finally, $Z_{0} = G$.
		
		From $Z(G/Z(G)) = N/Z(G)$, we deduce that $[N,G] = Z(G)$. Therefore,
		$$[Z_{0}, G] = G', \quad [Z_{1}, G] = Z(G), \quad [Z_{3}, G] = \{e\}.$$
		
		Next, note that
		$$Z_{1}/[Z_{1}, G] = N/Z(G) = Z(G/Z(G)) \cong C_{p} \times C_{p}$$
		(see \cite[Subsections~4.1 and 4.5]{RJ}). Furthermore, $|N| = p^{3}$ and $|G'| = p^{2}$. Hence, we get
		$$Z_{1}/[Z_{0}, G] = N/G' \cong C_{p}.$$
		
		Thus, by Theorems~\ref{Perlis-walker} and~\ref{thm:WedderburnGVZ}, the result follows.
	\end{enumerate}
	This completes the proof of Theorem~\ref{thm:WedderburnGVZp^5}.
\end{proof}

	\section{Primitive central idempotents}\label{sec:pci}
		
	In this section, we begin by recalling some basic notions concerning primitive central idempotents in rational group algebras, and then proceed to the proof of Theorem \ref{thm:pciGVZ}.  
	
	Let $G$ be a finite group. An element $e \in \mathbb{Q}G$ is called an \emph{idempotent} if $e^2=e$. A \emph{primitive central idempotent} in $\mathbb{Q}G$ is an idempotent that lies in the center of $\mathbb{Q}G$ and cannot be decomposed into a sum of two nonzero orthogonal idempotents, i.e., there do not exist $e', e'' \in \mathbb{Q}G$ such that $e = e' + e''$ and $e'e''=0$.  
	
	It is a standard fact that the collection of primitive central idempotents of $\mathbb{Q}G$ determines the Wedderburn decomposition of $\mathbb{Q}G$ into simple components. More precisely, if $e$ is a primitive central idempotent of $\mathbb{Q}G$, then $\mathbb{Q}Ge$ is a simple algebra. For $\chi \in \Irr(G)$, the element  
	\[
	e(\chi) := \frac{\chi(1)}{|G|} \sum_{g \in G} \chi(g) g^{-1}
	\]
	is a primitive central idempotent of $\mathbb{C}G$, and the set $\{e(\chi) : \chi \in \Irr(G)\}$ forms a complete set of primitive central idempotents of $\mathbb{C}G$. Moreover, for $\chi \in \Irr(G)$ one defines  
	\[
	e_{\mathbb{Q}}(\chi) := \sum_{\sigma \in \operatorname{Gal}(\mathbb{Q}(\chi)/\mathbb{Q})} e(\chi^{\sigma}),
	\]
	which gives a primitive central idempotent of $\mathbb{Q}G$.  
	
	For a subset $X \subseteq G$, we write  
	\[
	\widehat{X} := \frac{1}{|X|} \sum_{x \in X} x \in \mathbb{Q}G.
	\]
	If $N \trianglelefteq G$, define  
	\[
	\epsilon(G,N) :=
	\begin{cases}
		\widehat{G} & \text{if } N=G, \\[0.3em]
		\prod_{D/N \in M(G/N)} (\widehat{N}-\widehat{D}) & \text{otherwise},
	\end{cases}
	\]
	where $M(G/N)$ denotes the set of minimal nontrivial normal subgroups $D/N$ of $G/N$, with $D$ a subgroup of $G$ containing $N$.  
	
	In what follows, we determine a full set of primitive central idempotents of the rational group algebra of a nested GVZ $p$-group, where $p$ is an odd prime. We begin with a general lemma.
	
	\begin{lemma}\cite[Lemma 3.3.2]{JR}\label{lemma:pcilin}
		Let $G$ be a finite group and $\chi \in \lin(G)$ with kernel $N=\ker(\chi)$. Then 
		\begin{enumerate}
			\item $e_{\mathbb{Q}}(\chi)=\epsilon(G,N)$;
			\item $\mathbb{Q}G\epsilon(G,N) \cong \mathbb{Q}(\zeta_{|G/N|})$.
		\end{enumerate}
	\end{lemma}
	
	We are now ready to establish Theorem \ref{thm:pciGVZ}.  
	
	\begin{proof}[Proof of Theorem \ref{thm:pciGVZ}]
		Let $G$ be a finite nested GVZ $p$-group, with $p$ an odd prime. Suppose $\cd(G) = \{p^{\delta_i} : 0 \leq i \leq n,\; 0=\delta_0 < \delta_1 < \cdots < \delta_n\}$, and define $Z_{\delta_i} := Z(\chi)$ for some $\chi \in \Irr_{p^{\delta_i}}(G)$.  
		
		Take $\chi \in \nl(G)$ with $\chi(1)=p^{\delta_r}$ for some $r \in \{1,2,\dots,n\}$. By Lemma~\ref{lemma:GVZcharacter}, there exists $\mu \in \lin(Z_{\delta_r})$ such that  
		\[
		\chi = \chi_\mu(g) = 
		\begin{cases}
			p^{\delta_r}\mu(g) & \text{if } g \in Z(\chi); \\[0.3em]
			0 & \text{otherwise},
		\end{cases}
		\]
		where $\mu$ is the lift of $\bar{\mu}\in \Irr(Z_{\delta_r}/[Z_{\delta_r},G]\mid [Z_{\delta_{r-1}},G]/[Z_{\delta_r},G])$ to $Z_{\delta_r}$.  
		
		Let $N=\ker(\chi)$. Then $N=\ker(\chi_\mu)=\ker(\mu)$. Since $e(\chi)=e(\chi_\mu)=e(\mu)$, we compute
		\begin{align*}
			e_{\mathbb{Q}}(\chi) &= e_{\mathbb{Q}}(\chi_\mu) \\
			&= \sum_{\sigma \in \operatorname{Gal}(\mathbb{Q}(\chi_\mu)/\mathbb{Q})} e(\chi_\mu^{\sigma}) \\
			&= \sum_{\sigma \in \operatorname{Gal}(\mathbb{Q}(\mu)/\mathbb{Q})} e(\chi_{\mu^{\sigma}}) \\
			&= \sum_{\sigma \in \operatorname{Gal}(\mathbb{Q}(\mu)/\mathbb{Q})} e(\mu^{\sigma}) \\
			&= e_{\mathbb{Q}}(\mu) \\
			&= \epsilon(Z(\chi),N),
		\end{align*}
		using Lemma~\ref{lemma:pcilin}.  
		
		It is known that $G$ admits an irreducible character $\chi$ with $\chi(1) = |G/Z(\chi)|^{1/2}$ if and only if $\chi(g)=0$ for all $g \in G\setminus Z(\chi)$ (see \cite[Corollary 2.30]{I}). Thus, $\chi(1) = \chi_\mu(1) = p^{\delta_r} = |G/Z(\chi)|^{1/2}$. By Theorem~\ref{thm:WedderburnGVZ}, the simple component of $\mathbb{Q}G$ corresponding to $e_{\mathbb{Q}}(\chi)$ is  
		\[
		\mathbb{Q}G e_{\mathbb{Q}}(\chi) = \mathbb{Q}G \epsilon(Z(\chi),N) \;\cong\; M_{|G/Z(\chi)|^{1/2}} \!\left(\mathbb{Q}(\zeta_{|Z(\chi)/N|})\right).
		\]
		This completes the proof of Theorem \ref{thm:pciGVZ}.
	\end{proof}


\end{document}